\documentclass[11pt,a4paper]{article}
\usepackage{a4wide}
\usepackage{amsmath,amssymb}
\usepackage[english]{babel}
\usepackage{graphicx}
\usepackage[numbers]{natbib}
\usepackage{theorem}
\usepackage{bm}
\usepackage{paralist}
\usepackage{ifpdf}
\usepackage{hyperref}

\hypersetup{
    bookmarksopen=false,
    bookmarksnumbered=true,
    pdftitle={Analysis of delays at exhaustive traffic signals based on polling models},
    pdfauthor={M.A.A. Boon, E.M.M. Winands, I.J.B.F. Adan}
}
\ifpdf
  \hypersetup{colorlinks=true,linkcolor=black,urlcolor=black,citecolor=black,pdfpagemode=UseOutlines,plainpages=false,pdfpagelabels}
\else
  \hypersetup{colorlinks=false}
\fi

\theorembodyfont{\upshape}
\theoremheaderfont{\bfseries}

\newtheorem{theorem}{Theorem}[section]

\newtheorem{lemma}[theorem]{Lemma}
\newtheorem{remark}[theorem]{Remark}
\newtheorem{proposition}[theorem]{Proposition}
\newtheorem{corollary}[theorem]{Corollary}
\newenvironment{proof}{\par\textbf{Proof:}\\}{\hfill$\square$\par}
\numberwithin{equation}{section}

\newcommand{\ee}{\textrm{e}}
\newcommand{\dd}{\,\textrm{d}}

\newcommand{\E}{\mathbb{E}}
\renewcommand{\P}{\mathbb{P}}
\newcommand{\V}{\mathbb{V}\textrm{ar}}
\renewcommand{\O}{\mathcal{O}}
\newcommand{\C}{\textit{cv}}
\newcommand{\equaldist}{\,{\buildrel d \over =}\,}
\newcommand{\limdist}{\,{\buildrel d \over \rightarrow}\,}
\newcommand{\I}{\textit{I}}
\newcommand{\II}{\textit{II}}
\newcommand{\III}{\textit{III}}
\newcommand{\IV}{\textit{IV}}
\newcommand{\VV}{\textit{V}}
\newcommand{\VI}{\textit{VI}}
\newcommand{\VII}{\textit{VII}}
\newcommand{\VIII}{\textit{VIII}}
\newcommand{\IX}{\textit{IX}}
\newcommand{\X}{\textit{X}}
\newcommand{\XI}{\textit{XI}}
\newcommand{\XII}{\textit{XII}}

\providecommand{\href}[2]{#2}

\title{Delays at signalised intersections with exhaustive traffic control\footnote{The research was done in the framework of the BSIK/BRICKS project, and of the European Network of Excellence Euro-NF.}}
\author{M.A.A. Boon\footnotemark[2]\\\href{mailto:marko@win.tue.nl}{marko@win.tue.nl} \and I.J.B.F. Adan\footnote{\textsc{Eurandom} and Department of Mathematics and Computer Science, Eindhoven University of Technology, P.O. Box 513, 5600MB Eindhoven, The Netherlands}\\\href{mailto:iadan@win.tue.nl}{iadan@win.tue.nl} \and E.M.M. Winands \footnote{Department of Mathematics, Section Stochastics, VU University, De Boelelaan 1081a, 1081HV Amsterdam, The Netherlands}\\\href{mailto:emm.winands@few.vu.nl}{emm.winands@few.vu.nl}\and D.G. Down\footnote{McMaster University, 1280 Main Street West, Hamilton, ON L8S 4L7, Canada.}\\\href{mailto:downd@mcmaster.ca}{downd@mcmaster.ca}}

\date{July, 2011}

\begin{document}
\maketitle

\begin{abstract}
In this paper we study a traffic intersection with vehicle-actuated traffic signal control. Traffic lights stay green until all lanes within a group are emptied. Assuming general renewal arrival processes, we derive exact limiting distributions of the delays under Heavy Traffic (HT) conditions. Furthermore, we derive the Light Traffic (LT) limit of the mean delays for intersections with Poisson arrivals, and develop a heuristic adaptation of this limit to capture the LT behaviour for other interarrival-time distributions. We combine the LT and HT results to develop closed-form approximations for the mean delays of vehicles in each lane. These closed-form approximations are quite accurate, very insightful and simple to implement. 

\end{abstract}

\section{Introduction}\label{introduction}

Traffic signals play an important part in the infrastructure of towns and cities all over the world, and waiting before a red traffic light has become an unavoidable nuisance in everyday life. It is obvious that it is important to find optimal settings, i.e. green and red times, for signalised intersections. When evaluating the quality of traffic signal settings, the most commonly used criterion for optimality is a weighted average of the expected vehicle delays. Surprisingly, however, the most commonly used formulas for calculating the mean delays are based on major simplifications of the actual situation encountered in reality (see, e.g., \cite{handboekcrow,hcm2000}). On the other hand, for most realistic cases hardly any good alternatives exist. Typical traffic lights settings can be divided into three strategies: fixed cycle, vehicle-actuated, and traffic-actuated signals. The fixed-cycle policy for signalised intersections is the oldest strategy, defining fixed red, amber and green times. Vehicle-actuated traffic signals have flexible green phases with minimum and maximum green times. Detectors gather information about the presence of vehicles at the different traffic flows and use this information to determine whether the best option is to stay green, or switch to an all-red phase. Traffic-actuated signals are much like vehicle-actuated signals, with the additional possibility that actual queue lengths are determined. The combined information about all queue lengths can be used to create more sophisticated traffic signal settings. In the present paper we study a signalised intersection with vehicle-actuated, \emph{exhaustive} traffic control. The policy is exhaustive, because a green phase ends as soon as no vehicle is present in any flow that faces a green light. The main advantage of an exhaustive control policy is its efficiency, due to the fact that the traffic lights do not turn red until all vehicles in the corresponding flows have left. This implies that the exhaustive policy minimises the mean total amount of unprocessed work in the system, i.e., the total time required by all vehicles present at the intersection to discharge at the corresponding saturation flow rates. Newell \cite{newell98} argues that, for isolated intersections, an exhaustive control policy should be preferred over alternative strategies. Nevertheless, a well-known disadvantage of the exhaustive control policy is its unfairness with respect to vehicles in flows where relatively few vehicles arrive. For this reason, in practice often time-limited control policies are used.

Although the model in the present paper is vehicle-actuated, we will give a short literature review about fixed-cycle traffic signals first. Intersections based on a fixed cycle have been studied since a long time. One of the first, and perhaps still the most influential and practically applied papers, is written by Webster \cite{webster58} who analyses a fixed-cycle traffic-light queue. Although more sophisticated analyses have appeared throughout the years (cf. \cite{heidemann94,miller63,newell65,leeuwaarden06}), his formula for the mean delay is still used in most traffic engineering manuals. The present paper does not focus on settings based on a fixed cycle or on traffic-actuated signals, but on vehicle-actuated traffic signal control. Considering this is the most commonly used control type nowadays, it is surprising how little mathematical literature is available to analyse a typical, realistic vehicle-actuated signalised intersection. The earliest literature on vehicle-actuated systems dates from the early 1960s when Darroch et al. \cite{darroch64} analysed a system consisting of two intersecting Poisson traffic streams that are served exhaustively. A model with two lanes that does not assume Poisson input has been studied by Lehoczky \cite{lehoczky72}, who uses an alternating priority queueing model. Newell \cite{newell1} analyses an intersection with two one-way streets  using fluid and diffusion queueing approximations.
In \cite{newell2}, Newell and Osuna study a four-lane intersection where two opposite flows face a green light simultaneously. A variation of the two-lane intersection is introduced by Greenberg et al. \cite{greenberg88}, who analyse mean delays on a single rail line that has to be shared by trains arriving from opposite directions. This model is extended by Yamashita et al. \cite{yamashita06}, who study alternating traffic crossing a narrow one-lane bridge on a two-lane road. In many of the discussed papers traffic is modelled as fluid passing through the road. These types of approximations are fairly accurate when the traffic intensity is relatively high, but do not perform very well if there is a lot of variation in the arrival or departure processes. Vlasiou and Yechiali \cite{vlasiouyechiali07} use a different approach, modelling a traffic intersection as a polling system, consisting of multiple queues with an infinite number of servers visiting each queue simultaneously. A disadvantage that all of the aforementioned papers have in common, is that the methods can be applied only to situations that are significant simplifications of modern intersections encountered in practice. Most of the papers focus on two or four lanes only, and Newell and Osuna \cite{newell2} have written one of the few papers studying an intersection where multiple flows of vehicles can receive a green light simultaneously. Haijema and Van der Wal \cite{haijemavanderwal07} study a model that allows for multiple groups of different flows, but their approach uses a Markov Decision Problem (MDP) formulation, which does not lead to closed-form, transparent expressions for the mean delay that can be used for optimisation purposes.
Summarising, in the literature on vehicle-actuated traffic signals, either the models are very simplified versions of reality, or the resulting algorithm or expression to determine the mean delays is so complex that it cannot be implemented for real-life intersections. Additional advantages, apart from the comprehensiveness, of a closed-form expression for the mean delay, is that it is perfectly suitable for optimisation purposes, and that it can easily be adapted to extensions of the model discussed in the present paper.

The goal of this paper is to provide a comprehensive, novel analysis for traffic intersections with a vehicle-actuated, exhaustive control policy.
The model in the present paper is more realistic than most vehicle-actuated models, in the sense that we allow combinations of multiple flows to receive a green light simultaneously, we take more types of randomness into account (e.g., in interarrival times \emph{and} interdeparture times), and we do not limit ourselves to Poisson arrivals. The main contribution of the paper is twofold. Firstly, we capture the limiting behaviour of the model under Light Traffic (LT) and Heavy Traffic (HT) conditions. The HT results give insight into the system behaviour when the intersection becomes saturated, which makes them very usable by themselves. The LT results describe the system when it is hardly exposed to any traffic at all. These results are mostly interesting because they form an essential building block for the second main contribution of the paper, a closed-form approximation for the mean delay. This approximation is created using an interpolation between the LT and HT limits, which results in a comprehensive expression that is very insightful, simple to implement, and suitable for optimisation purposes.

We use a \emph{polling model} to describe the traffic intersection.
A polling model is a queueing system with multiple queues of customers, served by a single server in a cyclic order. The switch of the server from one queue to the next requires some (possibly random) time, and is called a switch-over time. An advantage of using a polling model, with customers representing the vehicles, is that we can model randomness in the interarrival times, but also in the service times, corresponding to the times between two successive vehicles as they pass the stop line. When considering the features of a polling model, it seems like the natural way to model a traffic intersection.
However, traffic intersections do exhibit features that have not been studied in the polling literature before, which impels us to considerably extend the queueing analysis of these systems. In particular, the feature that multiple queues (corresponding to the different traffic flows) should be allowed to receive service simultaneously has not been investigated yet in the huge literature on polling systems (for good surveys on polling systems and their applications, see, for example, \cite{boonapplications2011sorms,levysidi90,takagi1988qap,vishnevskiisemenova06}). The main reason why it is difficult to analyse a polling system with simultaneous service of multiple queues, is that the system loses the so-called \emph{branching-property}. We do not discuss this property in more detail here, but we just mention that Resing \cite{resing93} and Fuhrmann \cite{fuhrmann81} have shown that for polling models satisfying this property, performance measures like cycle time distribution and waiting time distributions can be obtained. If this branching property is, however, not satisfied, the corresponding polling model defies an exact analysis except for some special (two-queue or symmetric) cases. In the present paper we show that, despite the fact that the branching property is not satisfied, the delays can still be studied under LT and HT conditions. Furthermore, using these LT and HT results we propose an accurate closed-form approximation of the mean delay for a signalised intersection with vehicle-actuated, exhaustive traffic control.

The structure of the present paper is as follows. In the next section we present an outline of the model and describe in more detail how a traffic intersection can be modelled using a polling model with simultaneous service of multiple queues. The notation required for the remainder of the paper is also introduced in Section \ref{modelandnotation}. In Section~\ref{HTsection} we study the distribution of the delays under  Heavy Traffic conditions. This means that we increase the arrival intensities and, hence, the load of the system until it reaches the point of saturation. In Section \ref{LTsection} we study the behaviour of the system under Light Traffic conditions. It goes without saying that this situation is rather opposite to the previous section, and mean delays of vehicles in an (almost) empty system are analysed. Using the mean delays under LT conditions, and the mean delays under HT conditions, we develop interpolations between these two limits in Section \ref{interpolationssection}. These interpolations can be used as approximations for the mean delay under any system load. In Section \ref{numericalresults} we discuss the accuracy of the approximations and show numerical results for three intersections, located in The Netherlands. We finish with some conclusions and topics for further research.


\section{Model description and notation}\label{modelandnotation}

We model the traffic intersection as a polling system. A typical polling system consists of $N$ queues attended by a single server in a cyclic order. Each flow of vehicles, sometimes referred to as stream, corresponds to a queue in the polling system. Note that traffic approaching the intersection from the same direction, sharing one lane, but with different destinations (e.g., flow~9 in Figure \ref{fig:intersection}), is modelled as one single queue, whereas the same situation with two different lanes (e.g., flows 1 and 2) is modelled as two separate queues.
\begin{figure}[ht]
\begin{center}
\includegraphics[width=0.3\linewidth]{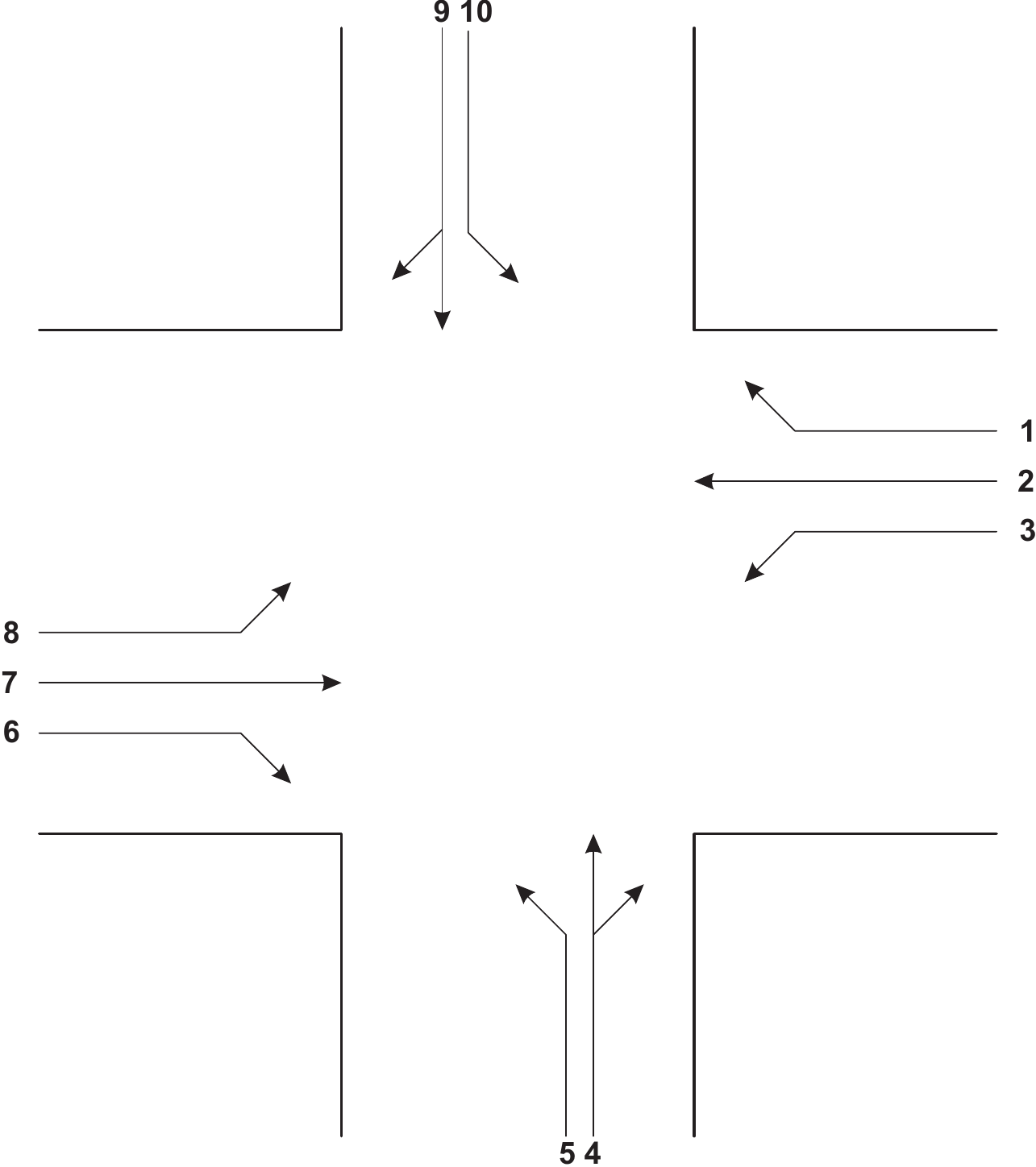}
\end{center}
\caption{An example of an intersection with traffic approaching from four directions.}
\label{fig:intersection}
\end{figure}
The main contribution of the present paper is that we extend the standard polling model by dividing the queues into groups that are served simultaneously, which turns out to complicate the queueing analysis significantly. We impose the restriction that each flow can be part of only one group. Note that the Highway Capacity Manual 2000 (HCM, \cite{hcm2000}) often does not use the terms ``flow'' or ``stream'', but prefers \emph{lane group}, defined as a set of lanes established at an intersection approach for separate capacity and level-of-service analysis. We will not adopt this terminology, except for a few times in this section, because it might be confused with a group (or combination) of non-conflicting flows that receive a green light simultaneously.
A \emph{cycle} consists of multiple phases, where groups of non-conflicting flows receive the right-of-way simultaneously. The switch-over times in the polling model correspond to the all-red times of a traffic intersection. Adopting the terminology used in the polling literature, we refer to the green periods as \emph{visit times}, i.e., the times that the server visits a group of queues. Similarly, the red period of a group of flows may be referred to as \emph{intervisit time} of that group.
In the present paper we assume that the control policy, which is called service discipline in polling literature, is exhaustive service. Since multiple queues are served simultaneously, exhaustive service implies that a green period only ends when \emph{all} queues in the corresponding group are empty. The $N$ flows of the intersection are divided into a number of groups, denoted by~$M$. Each group, say group $g=1,\dots,M$, consists of $N_g$ nonconflicting flows of vehicles. We assume that each flow belongs to exactly one group, so $\sum_{g=1}^M N_g = N$. Denote by $R_g$ the all-red time starting at the end of the green period of group $g$, denoted by $G_g$. The total all-red time in a cycle, denoted by $R$, equals $R = \sum_{g=1}^M R_g$. Throughout this paper, we assume that the all-red times $R_g$ are independent random variables. In reality, all-red times are generally deterministic, which simplifies many results obtained in this paper.

The times between two consecutive vehicles (in the same flow) crossing the stop line are generally called \emph{departure headways} and will be denoted by $B_i$, $i=1,\dots,N$. In the polling model, these headways correspond to service times of customers, but in the literature on traffic signals it is more common to use the reciprocal, $1/\E[B_i]$, which is referred to as discharge rate or saturation flow rate. An advantage of adopting a polling model, is that it allows for randomness in the headways, which is generally ignored in the literature on traffic signals. This randomness enables us to distinguish between slow and fast, or between big and small vehicles. An important aspect of our model, is that we make a distinction between the headways of queued vehicles, and vehicles approaching the intersection without queue in front of them. If a queue clears before the green period terminates, all vehicles that arrive during the remainder of this visit period pass through the system and experience no delay whatsoever. This assumption is quite common in the literature on traffic light queues to distinguish between headways of cars that need to accelerate and cars that arrive at full speed (see, e.g., \cite{vandenbroek06,leeuwaarden06}), but it is not common in the polling literature. Furthermore, we assume that the headways $B_i$ are independent of each other and of all other random variables in the model. This assumption is generally not satisfied in practice, because the first few vehicles might require a slightly longer time to accelerate. However, this can be circumvented by incorporating any systematic, additional delay of the first cars crossing the intersection after an all-red phase in the preceding all-red time ($R_{g-1}$ for vehicles belonging to group $g=1,\dots,M$).

We assume that the interarrival times of vehicles, corresponding to customers in the polling system, are independent, generally distributed random variables $A_i$, $i=1,\dots,N$. The arrival \emph{rates} are denoted by $\lambda_i = 1/\E[A_i]$, $i=1,\dots,N$. The main performance measures of interest in the present paper are the delays $W_i$ of vehicles in flow $i$. Note that for a \emph{queued} vehicle, the delay is the waiting time \emph{plus} $B_i$, whereas vehicles approaching a clear intersection during a green period experience no delay at all. We study the delay as a function of the total traffic load offered to the system, denoted by~$\rho$. The load of a particular flow, say flow $i$, is the product of the arrival rate and the mean departure headway: $\rho_i = \lambda_i \E[B_i]$. The total load is the sum of the loads of the different flows: $\rho=\sum_{i=1}^N \rho_i$. The HCM refers to $\rho_i$ as the \emph{flow ratio} of lane group $i$, being the ratio of the actual flow rate $(\lambda_i)$ to the saturation flow rate $(1/\E[B_i])$ for lane group $i$ at an intersection. Since, starting from the next section, we consider the delay as a function of $\rho$, the total amount of traffic offered to the system, we have to specify in more detail how the scaling takes place. The traffic intensity is varied by keeping the headways $B_i$ fixed, and scaling the interarrival times $A_i$ (or arrival rates $\lambda_i$). We denote \emph{unscaled} quantities by putting a hat on the scaled quantities. The settings that we call ``unscaled'' correspond to the situation where $\hat\rho = \sum_{i=1}^N \hat\rho_i = 1$, which has the advantage that $\hat\rho_i$ can be interpreted as the fraction of the total traffic load that is routed to flow $i$. This means that, if the total load equals~$\rho$, the actual flow ratio of flow $i$ is $\rho_i = \hat\rho_i \times \rho$. Hence, the unscaled interarrival times $\hat{A}_i$ are the interarrival times that lead to a system with load 1:
\[
\sum_{i=1}^N\hat\rho_i = \sum_{i=1}^N\hat\lambda_i\E[B_i]=\sum_{i=1}^N\frac{\E[B_i]}{\E[\hat A_i]} = 1.
\]
The other scaled/unscaled quantities follow from the relation $A_i=\hat{A}_i/\rho$. Another interesting quantity is the departure headway of an \emph{arbitrary} vehicle (in any flow), denoted by $B$. Since an arbitrary vehicle arrives in flow $i$ with probability $\lambda_i/\sum_{j=1}^N\lambda_j$, we have
\[
\E[B]=\frac{1}{\sum_{j=1}^N\lambda_j}\sum_{i=1}^N\lambda_i\E[B_i]=\frac{\rho}{\sum_{j=1}^N\lambda_j}=\frac{1}{\sum_{i=1}^N\hat\lambda_i}.
\]
In the remainder of the paper, we use the subscripts $\{g,j\}$ for flow $j$ \emph{within} group $g$. Without loss of generality, we order the flows within a group according to their flow ratios: $\rho_{g,1} > \rho_{g,2} \geq \dots \geq \rho_{g,N_g}$. For example, supposing that in Figure \ref{fig:intersection} flows 4 and 9 are part of group 1, and flow 4 has a higher flow ratio than flow 9, we use the notation $B_{1,1}=B_4$, and $B_{1,2}=B_9$. The flows with the highest flow ratios, flows $\{g,1\}$ for $g=1,\dots,M$, are called \emph{dominant flows}, or \emph{critical lane groups} in the HCM. Since we study the limiting behaviour of the intersection as it becomes saturated, the stability condition is an important issue for our analysis.
\begin{theorem}\label{stabilitytheorem}
The system is stable (i.e., the intersection is undersaturated) if and only if
\[\sum_{g=1}^M \rho_{g,1} < 1.\]
\begin{proof}
The proof is provided in Appendix \ref{stabilityproof}.
\end{proof}
\end{theorem}
Note that only the dominant flows play a role in this stability condition. We study the steady-state of stable systems only, so throughout the paper we assume that
\[0 \leq L\rho < 1,\]
where $L={\sum_{g=1}^M \hat\rho_{g,1}}$, i.e., the total relative load of the dominant flows. The quantity $L\rho$ is called the \emph{critical volume-to-capacity ratio} in the HCM, defined as the proportion of available intersection capacity used by vehicles in critical lane groups. Sometimes it is more comprehensive to use $L\rho$ instead of $\rho$, because its value is always between 0 and 1, for all stable intersections. It turns out to be convenient to introduce $\rho_{g,\bullet} = \sum_{j=1}^{N_g}\rho_{g,j}$ as the load of group $g$, and $\lambda_{g,\bullet} = \sum_{j=1}^{N_g}\lambda_{g,j}$ as the total arrival rate of group $g$.

Finally, the (equilibrium) residual length of a random variable $X$ is denoted by $X^\textit{res}$, with $\E[X^\textit{res}] = \E[X^2]/2\E[X]$. See, e.g., \cite{cohen82}, pp. 108 -- 115, for more information.

\section{Heavy traffic}\label{HTsection}

In the present section we study an intersection with exhaustive control policy under Heavy Traffic conditions. This means that we increase the load of the system until it reaches the point of saturation. From Theorem \ref{stabilitytheorem} we learn that the critical load for which the point of saturation is reached, is completely determined by the dominant flows in each group: the system becomes saturated as $\sum_{g=1}^M \rho_{g,1} \rightarrow 1$, which is equivalent to $L\rho \rightarrow 1$.
As the total load of the system, $\rho$, increases, the green times, the cycle times and waiting times become larger and will eventually grow to infinity. For this reason, we scale them appropriately and consider the scaled versions. Since we consider \emph{finite} all-red times, they become negligible compared to the waiting times as the load is increased. Polling systems under HT conditions have been studied by Coffman et al. \cite{coffman95,coffman98}, and by Olsen and Van der Mei \cite{olsenvdmei03,olsenvdmei05}. The key observation in these papers, is the occurrence of a so-called Heavy Traffic Averaging Principle (HTAP). When a polling system becomes saturated, two limiting processes take place. Let $V$ denote the total workload of the system. As the load offered to the system, $\rho$, tends to~1, the scaled total workload $(1-\rho)V$ tends to a Bessel-type diffusion. However, the work \emph{in each queue} is emptied and refilled at a faster rate than the rate at which the total workload is changing. This implies that during the course of a cycle, the total workload can be considered as constant, while the loads of the individual queues fluctuate like a fluid model. The HTAP relates these two limiting processes and provides expressions for the stationary distributions of the scaled cycle times, switch-over times, and waiting times. In order to derive the HT limits for traffic intersections, we introduce and analyse a novel \textit{fluid model}. Subsequently, we adapt and extend the HTAP to relate the results of this fluid model with the original traffic intersection model.

We introduce a fluid model, with work flowing in at constant rate $\lim_{\rho\rightarrow 1/L}\rho_{g,j}=\hat\rho_{g,j}/L$ for flow $\{g,j\}$. The all-red times are not considered in the fluid model, because under HT conditions they become negligible. A graphical illustration of the fluid model is presented in Figure~\ref{fig:htfluid}.
\begin{figure}[h!]
\begin{center}
\includegraphics[width=\linewidth]{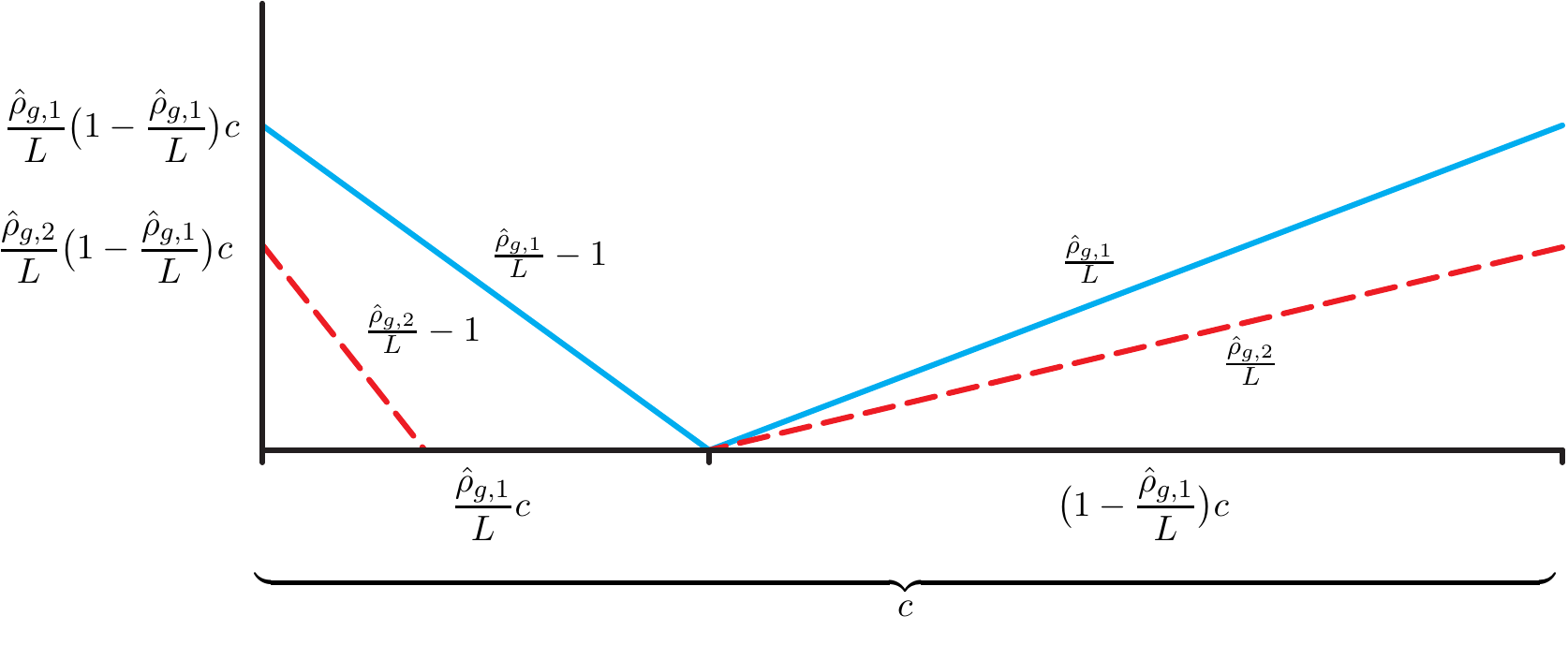}
\caption{Heavy traffic fluid limits.}
\label{fig:htfluid}
\end{center}
\end{figure}
On the horizontal axis, the course of a cycle with length $c$ is plotted. On the vertical axis, the scaled workloads in flows $\{g,1\}$ and $\{g,2\}$ are plotted. Because the length of the green periods is determined by the dominant flows, the system becomes saturated as $\sum_{g=1}^M \rho_{g,1} =L\rho\rightarrow 1$. A formal proof of this statement is provided in Appendix \ref{stabilityproof}. For the total load offered to the system, this translates to $\rho\rightarrow \frac{1}{L}$. We (arbitrarily) start the cycle at the moment that the traffic lights for the flows in group $g$ turn green. For now, during the first part of the analysis, we assume that the amount of work at the beginning of a cycle is fixed.
This implies that the length of a cycle, denoted by $c$, is also fixed because it is determined by the amount of work present at the beginning of the cycle. Throughout the cycle, work arrives with intensity $1/L$ and a fraction $\hat\rho_{g,j}$ is directed to flow $\{g,j\}$. During the green periods, work flows out of each stream at rate 1 as long as it is not empty. As soon as the stream is empty, it stays empty (hence, the work flows out at rate $\hat\rho_{g,j}/L$) until the end of the green period. As a consequence, although the  total amount of work in the system at the end of a cycle is back to the level of the beginning of a cycle, it varies throughout the cycle. However, if we consider the workload in \emph{dominant} flows only, the total workload remains constant, because it flows out at rate 1, and flows in at rate $\sum_{g}\hat\rho_{g,1}/L = 1$. This result follows directly from the observation that in the fluid model, the non-dominant flows do not contribute to the cycle length and, hence, do not influence the total workload in the dominant flows.

From the viewpoint of fluid in flow $\{g,j\}$, a cycle of length $c$ consists of three parts. During the first part, starting at the moment that the lights turn green, fluid starts to drain out of the system until flow $\{g,j\}$ is empty. The second part is the time between the emptying of flow $\{g,j\}$ and the moment that the dominant flow of the group, $\{g,1\}$ becomes empty and traffic lights turn red. The third part is the red period of group $g$. It is easily seen that the length of this third part is $\big(1-\hat\rho_{g,1}/L\big)c$, and the length of the first two parts together (the green period of group $g$) is $\frac{\hat\rho_{g,1}}{L}c$. Using Figure 2 we see that the length of the first part, with flow $\{g,j\}$ being non-empty, is $\frac{\hat\rho_{g,j}}{L}\frac{1-\hat\rho_{g,1}/L}{1-\hat\rho_{g,j}/L}c$. We denote the lengths of the three parts respectively by $P_j, P_1$, and $P_R$. The probability distribution of the delay of an arbitrary fluid particle arriving in flow $\{g,j\}$, denoted by $W^\textit{fluid}_{g,j}$, can now be computed.
\begin{theorem}\label{Wfluidtheorem}
\begin{equation}
W^\textit{fluid}_{g,j}
\equaldist \begin{cases}
0 & \qquad\text{w.p. }P_1/c, \\
U \times P_R,  & \qquad\text{w.p. }(P_j+P_R)/c,
\end{cases}\label{Wfluid}
\end{equation}
where $U$ is a uniformly distributed random variable on the interval $[0,1]$.
\begin{proof}
We condition on the arrival epoch of an arbitrary fluid particle in flow $\{g,j\}$. With probability $P_R/c$, the arrival takes place during the red period. The probability that the arrival takes place during the first part of the green period, when there is still other fluid present in flow $\{g,j\}$, is $P_j/c$. If the particle arrives during the second part of the green period, when there is no fluid present in flow $\{g,j\}$, its delay is 0. This happens with probability $P_1/c$. So we can write:
\begin{align*}
W^\textit{fluid}_{g,j} \equaldist
\begin{cases}
0 & \qquad\text{w.p. }P_1/c, \\
W_{g,j}^\textit{green} & \qquad\text{w.p. }P_j/c, \\
W_{g,j}^\textit{red} & \qquad\text{w.p. }P_R/c. \\
\end{cases}
\end{align*}
A particle arriving during the first part of the green period (with length $P_j$) has to wait until all the fluid in front of it has left the system. Let the uniform random variable $U_G$ denote the fraction of $P_j$ that has elapsed at the arrival epoch of this particle. Then the amount of fluid left in flow $\{g,j\}$ is $(1-\hat\rho_{g,j}/L)\times(1-U_G)P_j = \frac{\hat\rho_{g,j}}{L}\big(1-\hat\rho_{g,1}/L\big)(1-U_G)c$. Hence,
\begin{equation}
W_{g,j}^\textit{green} \equaldist \frac{\hat\rho_{g,j}}{L}(1-U_G)\left(1-\hat\rho_{g,1}/L\right)c.\label{wfluidgreen}
\end{equation}
The delay of a particle arriving during the red period can be analysed similarly. Let the uniform random variable $U_R$ denote the fraction of the red period that has elapsed at the arrival epoch of this particle. Then the amount of fluid present in flow $\{g,j\}$ is $\frac{\hat\rho_{g,j}}{L}\times U_R P_R$. The arriving particle has to wait for this amount of fluid to drain, after it has waited for the residual red period $(1-U_R)P_R$. Hence, we have
\begin{equation}
W_{g,j}^\textit{red} \equaldist \left(\frac{\hat\rho_{g,j}}{L} U_R+(1-U_R)\right)\big(1-\hat\rho_{g,1}/L\big)c.\label{wfluidred}
\end{equation}
If we study \eqref{wfluidgreen} and \eqref{wfluidred} more carefully, we see that $W_{g,j}^\textit{green}$ is uniformly distributed on the interval $[0, \frac{\hat\rho_{g,j}}{L}P_R]$ and  $W_{g,j}^\textit{red}$ is uniformly distributed on the interval $[\frac{\hat\rho_{g,j}}{L}P_R, P_R]$. Recall that $W^\textit{fluid}_{g,j} \equaldist W_{g,j}^\textit{green}$ with probability $P_j/c = \frac{\hat\rho_{g,j}/L}{1-\hat\rho_{g,j}/L}P_R/c=\frac{\hat\rho_{g,j}}{L}(P_j+P_R)/c$, and $W^\textit{fluid}_{g,j} \equaldist W_{g,j}^\textit{red}$ with probability $P_R/c = \left(1-\frac{\hat\rho_{g,j}}{L}\right)(P_j+P_R)/c$. This implies that $W^\textit{fluid}_{g,j}$ is uniformly distributed on $[0,P_R]$ with probability $(P_j+P_R)/c$, and it is 0 otherwise.
\end{proof}
\end{theorem}
Now that we have established the distribution of the delay of a particle in the fluid model, we can find the distribution of the scaled delay of a vehicle in the original model for the traffic intersection, under HT conditions.

\paragraph{Original model.} For ordinary polling models, the link between the fluid model and the polling model under HT conditions is the Heavy Traffic Averaging Principle. This principle states that the work \emph{in each queue} is emptied and refilled at a rate that is so much faster than the rate at which the total workload is changing, that during the course of a cycle, the total workload can be considered as constant, while the loads of the individual queues fluctuate like a fluid model.
A novel contribution of the present paper is the adaptation and extension of the HTAP for polling systems to the traffic intersection model.
Also in this model, when the system becomes saturated, the diffusion limits of the \emph{total} workload process and the workload in the individual flows can be related using the HTAP. The main difference is that in the fluid model for the traffic intersection, the total workload does not remain constant, but the total workload in the \emph{dominant} flows does. From the fluid model it has become clear that the non-dominant flows do not play a role in the length of the green periods, because the dominant flow in each group will always be the last flow that becomes empty. For this reason, we can ignore the non-dominant flows temporarily and focus on the workload in the dominant flows only. This turns our traffic light model into an ordinary polling model with exhaustive service. First, we define a random variable with a Gamma distribution as a random variable with probability density function
\[
f(t)=\frac{1}{\Gamma(\alpha)}\ee^{-\mu t}\mu^\alpha t^{\alpha-1}, \qquad t\geq 0,
\]
where $\Gamma(\alpha)=\int_0^\infty \ee^{-t}t^{\alpha-1}\dd t$. The positive parameters $\alpha$ and $\mu$ are respectively the scale and rate parameter.

When the load $\rho$ approaches~$1/L$, the system becomes overloaded and the queue lengths and waiting times tend to infinity. For this reason, we consider the \emph{scaled delay} $(1-L\rho)W_{g,j}$, which stays finite for $\rho\rightarrow 1/L$. Before we can formulate the main result of this section, the distribution of the scaled delay as the HT limit is approached, we need some lemmas. Note that the proofs of these lemmas all rely on the conjectures posed in \cite{olsenvdmei05}. Although these conjectures are widely accepted to be true, they have only been proven for systems consisting of two queues (cf. \cite{coffman95,coffman98}), systems with Poisson arrivals (cf. \cite{olsenvdmei03}), or for the \emph{means} rather than the complete distributions (cf. \cite{vdmeiwinands08}).
\begin{lemma}\label{lemmascaledworkHT}
Denote by $V_{\bullet,1}$ the amount of work \emph{in the dominant flows} of the intersection, at the beginning of a cycle. For $\rho \rightarrow 1/L$, $(1-L\rho)V_{\bullet,1}$ has a Gamma distribution with parameters $\alpha = 2\E[R]\delta/\sigma^2+1$ and $\mu = 2/\sigma^2$, where $\delta=\sum_{g=1}^M \frac{\hat\rho_{g,1}}{L}(1-\frac{\hat\rho_{g,1}}{L})/2$, and $\sigma^2=\sum_{g=1}^M \frac{\hat\lambda_{g,1}}{L}\left(\V[B_{g,1}]+\hat\rho_{g,1}^2\V[\hat A_{g,1}]\right)$.
\begin{proof}
By the HTAP, the total workload in the \emph{dominant flows} of the system may be regarded as unchanged over the course of a cycle.
So if we regard the system with dominant flows only, Lemma \ref{lemmascaledworkHT} follows directly from Olsen and Van der Mei \cite{olsenvdmei05}, Conjecture 1. More specifically, we use the special case of cyclic, exhaustive service to obtain the distribution of the scaled amount of work in the dominant flows.
\end{proof}
\end{lemma}
Before continuing, it is important to realise the interpretation of $\delta$. Assume, just like in Figure \ref{fig:htfluid}, that $c$ is the length of a cycle, starting at the beginning of green period $G_g$. Then the workload in flow $\{g,1\}$ at the beginning of the cycle is $\frac{\hat\rho_{g,1}}{L}\big(1-\frac{\hat\rho_{g,1}}{L}\big)c$. The flow is empty at the moment that a fraction $\frac{\hat\rho_{g,1}}{L}$ of the cycle has passed, and at the end of the cycle it has reached the same level as at the beginning. The mean workload in flow $\{g,1\}$ during this cycle is
$
\frac12\frac{\hat\rho_{g,1}}{L}\big(1-\frac{\hat\rho_{g,1}}{L}\big)c$.
A summation over all dominant flows shows that $\delta c$ is the mean total workload of the dominant flows during the course of one cycle of length $c$

Given the scaled amount of work in the dominant queues at the start of a green period of group $g$, we can derive the distribution of the scaled \emph{cycle time} $(1-L\rho)C$. In fact, we consider the \emph{length-biased} (or time-averaged) scaled cycle time $(1-L\rho)\bm{C}$. If a random variable $X$ has probability density function $F_X(x)$, then we define the length-biased random variable $\bm{X}$ as a random variable with probability density function $f_{\bm{X}}(x)=x f_X(x)/\E[X]$. From renewal theory, we know that the length-biased cycle length accounts for the fact that an arbitrary arriving vehicle arrives with a higher probability during a long cycle, than during a short one.
The length of a cycle depends on the amount of work at the beginning of that cycle. Denote by $C(x)$ the length of a cycle, given that a total amount of $x$ work is present in the dominant flows. We are now ready to formulate the second lemma, needed to find the distributions of the scaled delays under HT conditions.
\begin{lemma}\label{lemmascaledintervisitHT}
Denote by $\bm{I}_g$ the length-biased red-time (or intervisit time) of group $g$, $g=1,\dots,M$. For $\rho \rightarrow 1/L$, we find that $(1-L\rho){\bm{I}_g}$ converges in distribution to a random variable having a Gamma distribution with parameters $\alpha = 2\E[R]\delta/\sigma^2+1$ and $\mu_g:=2\delta/\big(\sigma^2(1-\hat\rho_{g,1}/L)\big)$.
\begin{proof}
The proof proceeds along the same lines as the argument given to support Conjecture 2 in \cite{olsenvdmei05}. It uses Lemma \ref{lemmascaledworkHT} and the fact that, due to the averaging principle, the scaled workload in the dominant flows remains effectively constant. In steady-state, we have the following relation:
\[\delta C(x) = x .\]
This relation can easily be verified graphically from Figure \ref{fig:htfluid}, because the total amount of work in the dominant queues remains constant throughout the course of a cycle. Given a cycle length of $c$, the total amount of work in the dominant queues is $\sum_{g=1}^M \frac{\hat\rho_{g,1}}{L}\big(1-\frac{\hat\rho_{g,1}}{L}\big)c/2 = \delta c$. Hence, given an amount of work $x$ in the dominant flows, the cycle time is $C(x)=x/\delta$. Now we use Lemma \ref{lemmascaledworkHT}, which states that, in the HT limit, the scaled workload in the dominant flows has a Gamma distribution with parameters $\alpha$ and $\mu$. This implies that for $\rho\rightarrow 1/L$, the scaled, length-biased cycle time $(1-L\rho)\bm{C}$ follows a Gamma distribution with, again, scale parameter $\alpha$, but with rate parameter $\mu\delta$.
The distributions of the scaled length-biased intervisit times, denoted by $(1-L\rho)\bm{I}_g$ for group $g=1,\dots,M$, can now be determined. The intervisit time of group $g$ is the time that the signals in group $g$ are red.
Given that $x$ is the amount of work present at the dominant queues at the beginning of a cycle, the intervisit time conditioned on $x$ is obviously
\[
I_g(x) = C(x)\big(1-\frac{\hat\rho_{g,1}}{L}\big).
\]
The limiting distribution of $(1-L\rho){\bm{I}_g}$ now readily follows from the limiting distribution of $(1-L\rho){\bm{C}}$.
\end{proof}
\end{lemma}

\noindent
Finally, we formulate the main result of the present section.
\begin{theorem}\label{theoremscaleddelayHT}
As $\rho \uparrow 1/L$, the scaled delay is 0 with probability $\frac{\hat\rho_{g,1}-\hat\rho_{g,j}}{L-\hat\rho_{g,j}}$, and it is the product of a uniformly distributed random variable on $[0,1]$, denoted by $U$, and a random variable $\Gamma_I$ having the same distribution as the limiting distribution of $(1-L\rho)\bm{I}_g$, with probability $\frac{1-\hat\rho_{g,1}/L}{1-\hat\rho_{g,j}/L}$:
\begin{align}
(1-L\rho)W_{g,j} \limdist& \begin{cases}
0 & \qquad\text{w.p. }\frac{\hat\rho_{g,1}-\hat\rho_{g,j}}{L-\hat\rho_{g,j}}, \\
U \times \Gamma_I,  & \qquad\text{w.p. }\frac{1-\hat\rho_{g,1}/L}{1-\hat\rho_{g,j}/L},
\end{cases}
\label{scaleddelayHT}
\end{align}
for $\rho\rightarrow 1/L$.
\begin{proof}
A combination of Theorem \ref{Wfluidtheorem} and Lemma \ref{lemmascaledintervisitHT} yields the desired result. The scaled intervisit time $(1-L\rho){\bm{I}_g}$ converges in distribution to a random variable having a Gamma distribution with parameters $\alpha$ and $\mu_g$. The HTAP states that we can simply replace $P_R$, the deterministic red time in the fluid model, in \eqref{Wfluid} by the scaled, length-biased red time in the original model, $(1-L\rho){\bm{I}_g}$, because the random variables $U$ and $\bm{I}_g$ are independent.
\end{proof}
\end{theorem}

For the approximations developed in Section \ref{interpolationssection}, the \emph{mean} scaled delay for $\rho\rightarrow 1/L$ will be used.
\begin{corollary}
\begin{equation}
\lim_{\rho\rightarrow \frac1L}(1-L\rho)\E[W_{g,j}] = \frac{\left(1-\hat\rho_{g,1}/L\right)^2}{1-\hat\rho_{g,j}/L}  \left(\frac{\E[R]}{2} + \frac{\sigma^2}{\delta}  \right),\label{EWht}
\end{equation}
where $\delta=\sum_{g=1}^M \frac{\hat\rho_{g,1}}{L}(1-\frac{\hat\rho_{g,1}}{L})/2$, and $\sigma^2=\sum_{g=1}^M \hat\lambda_{g,1}\left(\V[B_{g,1}]+\hat\rho_{g,1}^2\V[\hat A_{g,1}]\right)$.
\begin{proof}
The result immediately follows from \eqref{scaleddelayHT}.
\end{proof}
\end{corollary}
Note that any possible variations in the all-red times $R_g$ have no influence on the mean scaled delay under HT conditions.

\begin{remark}
For Poisson arrivals $\sigma^2$ can be simplified to $\sigma^2 = \frac{\E[B_{\bullet,1}^2]}{\E[B_{\bullet,1}]}$, where $B_{\bullet,1}=\frac{\sum_{g=1}^M\hat\lambda_{g,1}B_{g,1}}{\sum_{g=1}^M\hat\lambda_{g,1}}$ is the headway of an arbitrary vehicle arriving in a dominant flow.
\end{remark}

\begin{remark}[Convergence to the HT limit]
Although the distribution of the scaled delay in the HT limit $\rho\uparrow 1/L$ is exact, it is interesting to know how fast this limiting distribution is approached. Unfortunately, the analysis does not provide any insight in the speed at which the scaled delay converges to this limiting distribution, so we have to resort to simulations. Denote by $p_{g,j}$ the steady-state probability that flow $\{g,j\}$ is the flow from which the last departure takes place before the corresponding green phase $G_g$ ends. This probability depends on $\rho$, the total load of the system. We have shown in this section that, in the limiting situation, the dominant flow of each group is \emph{always} the last flow to become empty in this group. For $\rho < 1/L$, the fraction of green periods in which the last departure indeed takes place from the dominant flow becomes smaller. In fact, in the LT limit one can easily calculate the exact value of $p_{g,j}$. In LT, the probability that the last vehicle in group~$g$ departs from flow $\{g,j\}$ is proportional to the arrival rates of the flows in this group. Summarising:
\begin{align*}
\lim_{\rho\uparrow 1/L} p_{g,j} &= \begin{cases}
1 & \qquad\textrm{if }j=1,\\
0 & \qquad\textrm{if }j> 1,
\end{cases}\\
\lim_{\rho\downarrow 0} p_{g,j} &= \frac{\hat\lambda_{g,j}}{\hat\lambda_{g,\bullet}}.
\end{align*}
This implies that we can regard $p_{g,j}$ as a measure for how close the distribution of the scaled delay is to the limiting distribution. In Figure \ref{accuracyHTpictures} we show an example that illustrates how $p_{g,j}$ might depend on $\rho$. The specific intersection chosen for this figure, has two groups of three flows each. The traffic intensities of the flows within each group are relatively close to each other. In this situation, it takes rather long before the probabilities $p_{g,j}$ converge to their limiting values for $\rho\rightarrow 1/L$. For example, in Figure \ref{accuracyHTpictures} one can see that at $L\rho=0.9$, the probability $p_{2,1}$ (corresponding to flow 6) is less than $0.7$. The fact that this value is still quite different from its limiting value 1, has a direct impact on the distribution of the scaled delay. The (simulated) mean scaled delay $(1-L\rho)\E[W_{2,1}]$ at $L\rho=0.9$ is $4.5$, whereas $\lim_{L\rho\uparrow1}(1-L\rho)\E[W_{2,1}]=3.5$. For this reason, one has to be careful when applying the results for $\rho\approx1/L$ to situations with smaller loads. In Section \ref{numericalresults}, we continue the discussion on the numerical accuracy. For more information about the intersection which is used for Figure \ref{accuracyHTpictures}, see Scenario \VV\ in Example 1.
\begin{figure}[!ht]
\mbox{}\hfill
\parbox{0.45\textwidth}{\centering
\includegraphics[width=\linewidth]{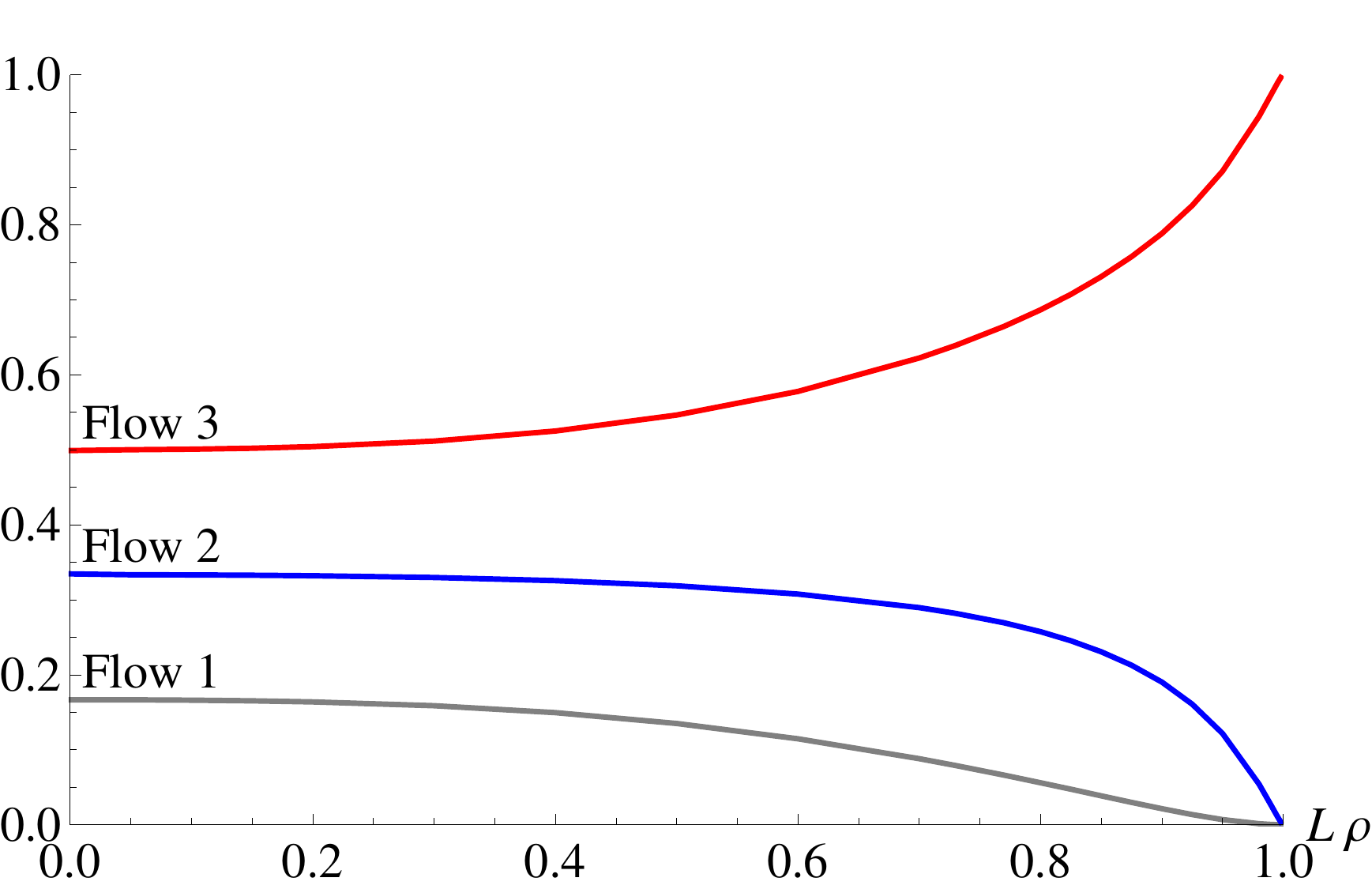}\\
Group 1
}
\hfill
\parbox{0.45\textwidth}{\centering
\includegraphics[width=\linewidth]{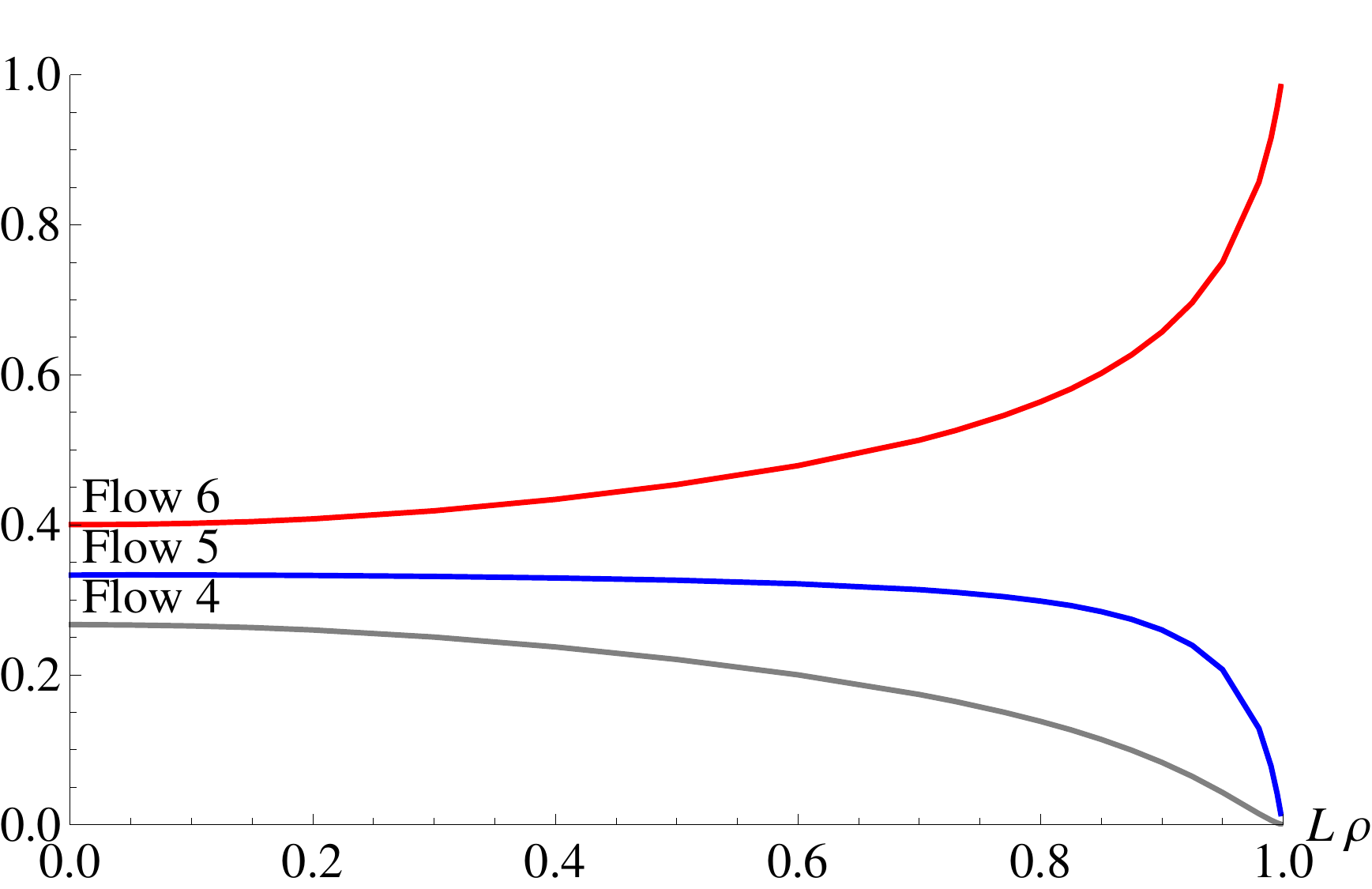}\\
Group 2
}
\hfill\mbox{}
\caption{Simulated fractions of the green periods where the corresponding flows have been the last in their groups from which a vehicle departs before the traffic lights turn red.}
\label{accuracyHTpictures}
\end{figure}
\end{remark}

\section{Light traffic}\label{LTsection}

In the present section we study the delay of vehicles arriving at the traffic intersection under Light Traffic conditions, i.e., for $\rho \downarrow 0$. In the first part of the section, we find an expression for the \emph{mean} delay of vehicles, assuming Poisson arrivals. Subsequently, we heuristically adapt this expression to find an approximation for the mean delay for intersections with general renewal arrivals. The analysis of the present section follows the lines of \cite{boonapprox2009}, where the derivation of the LT limit of mean waiting times for polling systems is discussed, with modifications for the specific traffic intersection model.

For the first part in this section we assume that the arrival processes in all flows are Poisson. Under this assumption, the LT limit of the mean delay can be determined by conditioning on the arrival epoch of an arbitrary vehicle in, say, flow $\{g,j\}$. 
By the LT limit of $\E[W_{g,j}]$, we mean the expression for $\E[W_{g,j}]$ as a function of the load in the system, $\rho$, up to $\O(\rho)$ terms for $\rho\downarrow0$. In the following theorem we formulate the main result for the system with Poisson arrivals.
\begin{theorem}\label{EWltPoissonTheorem}
Under the assumption of Poisson arrivals, the LT limit of the mean delay of an arbitrary type $\{g,j\}$ vehicle is
\begin{align}
\E[W_{g,j}^{\textit{LT,Poisson}}] =\,& \rho_{g,j} \left(\E[B_{g,j}^\textit{res}] +\E[B_{g,j}]\right)
+ \sum_{m=g-M+1}^{g-1} \sum_{k=1}^{N_m} \rho_{m,k}\left(\E[B_{m,k}^\textit{res}] + \sum_{l=m}^{g-1}\E[R_l]  +\E[B_{g,j}]\right)\nonumber\\
+\,& \sum_{m=g-M}^{g-1} \frac{\E[R_m]}{\E[R]}\left(E[R_m^{\textit{res}}]\Big(1-\rho+2\sum_{k=m+1}^{g-1}\rho_{k,\bullet} + \rho_{g,j} \Big) +\right.\nonumber\\
\,& \hspace{1cm}\left.
\sum_{k=g-M}^{m-1}\E[R_k] \Big(\sum_{l=m+1}^{g-1}\rho_{l,\bullet}+\rho_{g,j}\Big)
+\sum_{k=m+1}^{g-1}\E[R_k] \Big(1-\rho+\sum_{l=m+1}^{g-1}\rho_{l,\bullet}\Big)+(1-\rho)\E[B_{g,j}]\right)\nonumber\\
+\,&\O(\rho^2).
\label{EWltPoisson}
\end{align}
\end{theorem}
\begin{proof}
The proof is essentially a Mean Value Analysis (MVA) that ignores all $\O(\rho^2)$ terms, and focusses mainly on the amount of work instead of number of vehicles.
The first step is to condition on the arrival epoch of a vehicle in flow $\{g,j\}$. A cycle consists of the green phases $G_g$ and all-red phases $R_g$, for $g=1,\dots,M$. At the beginning of green time $G_g$, the probability that \emph{one} vehicle has arrived in any of the flows in group $g$ during the preceding red time (or: intervisit time) $I_g$, is $\lambda_{g,\bullet}\E[I_g] + \O(\lambda_{g,\bullet}^2)$. The probability that more than one vehicle has arrived in any flow of group $g$ is $\O(\lambda_{g,\bullet}^2)$ and therefore negligible in LT analysis. Therefore, we have that
\begin{equation}
\E[G_g] = \sum_{j=1}^{N_g}\rho_{g,j}\E[I_g] + \O(\rho_{g,j}^2) = \sum_{j=1}^{N_g}\rho_{g,j}\E[R] + \O(\rho^2).\label{evLT}
\end{equation}
The mean cycle time is
\begin{equation}
\E[C] = \sum_{g=1}^{M}\E[G_g] + \E[R] = \E[R]\left(1+\rho  \right) + \O(\rho^2).\label{ecLT}
\end{equation}
Now we can find the LT limit of $\E[W_{g,j}]$ by conditioning on the arrival epoch of an arbitrary customer. The probability that an arbitrary vehicle arrives during the all-red phase $R_k$, for $k=1,\dots,M$, is
\[\P(\text{arrival during $R_k$}) = \frac{\E[R_k]}{\E[C]} = \frac{\E[R_k]}{\E[R]}(1-\rho)+\O(\rho^2).\]
The probability that an arbitrary vehicle arrives during the green phase $G_k$, for $k=1,\dots,M$, is
\[\P(\text{arrival during $G_k$}) = \frac{\E[G_k]}{\E[C]} = \frac{\E[G_k]}{\E[R]}(1-\rho)+\O(\rho^2) = \sum_{l=1}^{N_k}\rho_{k,l} +\O(\rho^2).\]
If a vehicle arrives during $G_k$, the vehicle crossing the intersection at that moment is a type $\{k,l\}$ vehicle with probability $\frac{\rho_{k,l}}{\rho_{k,\bullet}}$, for $l=1,\dots,N_k$. This means that we can formulate the mean delay of a type $\{g,j\}$ vehicle as follows.
\begin{equation}
\E[W_{g,j}] = \sum_{k=1}^M \left(\frac{\E[R_k]}{\E[R]}(1-\rho)\E[W_{g,j}^{(R_{k})}]+\sum_{l=1}^{N_k} \rho_{k,l}\E[W_{g,j}^{(G_{k,l})}] \right) +\O(\rho^2),\label{EWltPoissonDecomp}
\end{equation}
where $\E[W_{g,j}^{(R_{k})}]$ is the mean delay of a $\{g,j\}$ vehicle that arrives during $R_k$, and $\E[W_{g,j}^{(G_{k,l})}]$ is the mean delay of a $\{g,j\}$ vehicle that arrives during the service of a $\{k,l\}$ vehicle. We will study these conditional mean delays in more detail.

Firstly, we note that the mean delay of a $\{g,j\}$ vehicle arriving while another vehicle of the same type is crossing the intersection, is simply the residual headway of the crossing vehicle, plus the departure headway of the vehicle itself:
\begin{align}
\E[W_{g,j}^{(G_{g,j})}] = \E[B_{g,j}^\textit{res}] + \E[B_{g,j}] +\O(\rho), \qquad g=1,\dots,M; j=1,\dots,N_g.\label{ewLT1}
\end{align}
The $\O(\rho)$ terms are of no interest in \eqref{ewLT1}, because the probability of an arrival during a green period is an $\O(\rho)$ term itself. Secondly, a $\{g,j\}$ vehicle arriving while another vehicle of the same group, but \emph{not} of the same flow, is crossing the intersection experiences no delay at all, because the probability that another vehicle of the same type is present at the intersection is $\O(\rho)$:
\begin{align}
\E[W_{g,j}^{(G_{g,l})}] = 0 +\O(\rho), \qquad g=1,\dots,M; j=1,\dots,N_g; l\neq j.\label{ewLT2}
\end{align}

The mean delay of a $\{g,j\}$ vehicle arriving while a vehicle of \emph{another} group is crossing the intersection, is composed of the mean residual headway of that other vehicle, plus all subsequent all-red times until the vehicle itself can start crossing the intersection, plus its own headway. Note that the intermediate green times are negligible, because their total length is $\O(\rho)$.
\begin{align}
\E[W_{g,j}^{(G_{k,l})}] = \E[B_{k,l}^\textit{res}]+\sum_{i=k}^{g-1}\E[R_i]  +\E[B_{g,j}] +\O(\rho), \qquad g\neq k.\label{ewLT3}
\end{align}
Note that the sum in \eqref{ewLT3} has to be taken cyclic over the all-red periods between the green times of groups~$k$ and~$g$. The mean delay of vehicles arriving during all-red times is slightly more complicated. We have to include all $\O(\rho)$ terms now, because the probability of an arrival during an all-red period is not $\O(\rho)$. The mean delay of a $\{g,j\}$ vehicle arriving during all-red period $R_m$ consists of the residual all-red period $R_m^\textit{res}$, plus the green \emph{and} all-red periods between group $m$ and group $g$. In more detail, the mean delay is composed of:
\begin{compactenum}
\item The departure headways of all vehicles that arrive in groups $m+1,\dots,g-1$ \emph{and} in flow $\{g,j\}$ during the all-red times $R_g, \dots, R_{m-1}$, and the elapsed part of $R_m$ at the arrival epoch, denoted by $R_m^\textit{past}$;\label{firstenum}
\item The residual red time $R_m$, denoted by $R_m^\textit{res}$, plus the headways of all vehicles arriving in groups $m+1,\dots,g-1$ during this residual red time;
\item The all-red times $R_{m+1},\dots,R_{g-1}$ plus the headways of the vehicles that arrive during these all-red times and will be served between $R_m$ and the green period of group $g$;
\item The headway of the arriving vehicle itself.\label{lastenum}
\end{compactenum}
Note that $R_m^\textit{past}$ is the same, in distribution, as $R_m^\textit{res}$. This leads to the following expression for $\E[W_{g,j}^{(R_m)}]$, where the terms $\eqref{ewLT4firsteqn}-\eqref{ewLT4}$ correspond to items $\ref{firstenum}-\ref{lastenum}$:
\begin{align}
\E[W_{g,j}^{(R_m)}] &= \left(\sum_{k=g-M}^{m-1}\E[R_k]+E[R_m^{\textit{past}}]\right)\left(\sum_{l=m+1}^{g-1}\rho_{l,\bullet}+\rho_{g,j}\right)\label{ewLT4firsteqn}\\
&+ E[R_m^{\textit{res}}]\left(1 + \sum_{l=m+1}^{g-1}\rho_{l,\bullet}\right)\\
&+ \sum_{k=m+1}^{g-1}\E[R_k] \Big(1+\sum_{l=m+1}^{g-1}\rho_{l,\bullet}\Big)\\
&+ \E[B_{g,j}].\label{ewLT4}
\end{align}
Combining Equations \eqref{evLT}--\eqref{ewLT4} completes the proof of Theorem \ref{EWltPoissonTheorem}.
\end{proof}

\begin{remark}
Equation \eqref{ewLT2} will turn out to be the only place in the LT analysis where the assumption that an empty flow stays empty during the remainder of the green period plays a role. Without this assumption, we would have that $\E[W_{g,j}^{(G_{g,l})}] = \E[B_{g,j}] +\O(\rho)$ for $l\neq j$. The HT analysis in the previous section would not change at all. This means that the assumption, perhaps surprisingly, does not have much impact at all on the real and the approximated mean delays. In Section \ref{numericalresults}, Example 3, we discuss in more detail the impact of the assumption that vehicles arriving at an empty flow during a green phase do not experience any delay.
\end{remark}

The main result of this section is an adaptation of \eqref{EWltPoisson}, which can be used as an approximation for the LT limit of the mean delay for general renewal arrivals.
\begin{corollary}\label{lttheorem}
An approximation for the LT limit of the mean delay for vehicles in flow $\{g,j\}$ for a traffic intersection with general renewal arrivals and deterministic all-red times is:
\begin{align}
\E[W_{g,j}^{\textit{LT}}] \approx\,& \rho_{g,j} \left(\E[\hat{A}_{g,j}]\hat{g}_{g,j}(0)-1\right) \E[B_{g,j}^\textit{res}])+\rho\E[B^\textit{res}]+\E[B_{g,j}]-\sum_{k\neq j} \rho_{g,k}\left(\E[B^\textit{res}_{g,k}]+\E[B_{g,j}]\right)\nonumber\\
+\,& \frac12(1+\rho+\rho_{g,j}-2\rho_{g,\bullet})R.\label{EWltDeterministic}
\end{align}
\end{corollary}
We derive approximation \eqref{EWltDeterministic} by adapting \eqref{EWltPoisson} to create an LT approximation for the case of general renewal arrivals. This adaptation is similar as in \cite{boonapprox2009}, based on the observation that the first term in \eqref{EWltPoisson}, $\rho_{g,j} \E[B_{g,j}^\textit{res}]$, is the LT limit of the mean waiting time (excluding the service time) in an $M/G/1$ queue in isolation:
\begin{equation}
\E[W_{M/G/1}] = \rho \E[B^\textit{res}]+\O(\rho^2).\label{EWwhittPoisson}
\end{equation}
We obtain an approximation for the mean waiting time for general renewal arrivals by replacing this term with the LT limit for the $GI/G/1$ queue from Whitt~\cite{whitt89}:
\begin{equation}
\lim_{\rho\downarrow0} \frac{\E[W_{GI/G/1}]}{\rho}= \frac{1+\C_{B}^2}{2}\E[\hat{A}]\hat{g}(0)\E[B],\label{EWwhitt}
\end{equation}
where $\C_{B}^2$ is the Squared Coefficient of Variation (SCV) of the service times, and $\hat{g}(t)$ is the density of the interarrival times $\hat{A}$ at $\rho=1$.
Our approximation is based on the approximative assumption that the Fuhrmann-Cooper decomposition (cf. \cite{fuhrmanncooper85}) holds for our system with general renewal arrivals. The Fuhrmann-Cooper decomposition for the waiting times of customers in queue~$i$ of a polling system with Poisson arrivals and exhaustive service in queue $i$ states that the waiting time of an arbitrary type $i$ customer is the sum of two independent random variables. One of these random variables is the waiting time of a customer in the corresponding $M/G/1$ queue in isolation (as if it would not have been part of a polling system), and the other random variable is the residual intervisit time of queue $i$:
\[
W_i \equaldist W_{i,M/G/1} + I_i^\textit{res}. 
\]
If we combine this Fuhrmann-Cooper decomposition with \eqref{EWltPoisson}, \eqref{EWwhittPoisson}, and \eqref{EWwhitt}, we obtain the following approximation for the mean delay (i.e., waiting time \emph{plus} headway)  of a vehicle in flow~$\{g,j\}$ for an intersection with general renewal arrivals:
\begin{align}
\E[W_{g,j}^{\textit{LT}}] \approx\,& \rho_{g,j} \left(\E[\hat{A}_{g,j}]\hat{g}_{g,j}(0) \E[B_{g,j}^\textit{res}] +\E[B_{g,j}]\right)
+ \sum_{m=g-M+1}^{g-1} \sum_{k=1}^{N_m} \rho_{m,k}\left(\E[B_{m,k}^\textit{res}] + \sum_{l=m}^{g-1}\E[R_l]  +\E[B_{g,j}]\right)\nonumber\\
+\,& \sum_{m=g-M}^{g-1} \frac{\E[R_m]}{\E[R]}\left(E[R_m^{\textit{res}}]\Big(1-\rho+2\sum_{k=m+1}^{g-1}\rho_{k,\bullet} + \rho_{g,j} \Big) +\right.\nonumber\\
\,& \hspace{1cm}\left.
\sum_{k=g-M}^{m-1}\E[R_k] \Big(\sum_{l=m+1}^{g-1}\rho_{l,\bullet}+\rho_{g,j}\Big)
+\sum_{k=m+1}^{g-1}\E[R_k] \Big(1-\rho+\sum_{l=m+1}^{g-1}\rho_{l,\bullet}\Big)+(1-\rho)\E[B_{g,j}]\right).
\label{EWlt}
\end{align}
This expression can be written into a more compact form. After some straightforward (but tedious) rewriting, it can be shown that \eqref{EWlt} reduces to:
\begin{align}
\E[W_{g,j}^{\textit{LT}}] \approx\,& \rho_{g,j} \left(\E[\hat{A}_{g,j}]\hat{g}_{g,j}(0)-1\right) \E[B_{g,j}^\textit{res}]+\rho \E[B^\textit{res}]+\E[B_{g,j}]-\sum_{k\neq j} \rho_{g,k}\left(\E[B^\textit{res}_{g,k}]+\E[B_{g,j}]\right)\nonumber\\
+\,& (1-\rho+\rho_{g,j})\E[R^\textit{res}]+(\rho -\rho_{g,\bullet})\E[R]+\frac{1}{\E[R]}\sum_{m=g-M}^{g-1}\left(\sum_{k=m+1}^{g-1}\rho_{k,\bullet}\right)\V[R_m].
\label{EWltnotdeterministic}
\end{align}
For deterministic all-red times, we have $\E[R^\textit{res}]=R/2$ and $\V[R_g]=0$ for $g=1,\dots,M$. Carrying out these substitutions in \eqref{EWltnotdeterministic} leads to expression \eqref{EWltDeterministic}.

\begin{remark}
A simplification can be made to replace the expression $\E[\hat{A}_{g,j}]\hat{g}_{g,j}(0)$. In \cite{boonapprox2009,whitt89} it is shown that a good approximation for this term is
\[
\E[\hat{A}_i]\hat{g}_i(0) \approx \begin{cases}
2\frac{\C_{A_i}^2}{\C_{A_i}^2+1} &\qquad\text{ if } \C_{A_i}^2>1,\\
\left(\C_{A_i}^2\right)^4 &\qquad\text{ if } \C_{A_i}^2\leq1,
\end{cases}
\]
where $\C_{A_i}^2$ is the squared coefficient of variation of $A_i$ (and, hence, also of $\hat{A}_i$). Note that this simplification results in an approximation that requires only the first two moments of each input variable (i.e., headways, all-red times, and interarrival times).
\end{remark}

\begin{remark}[Convergence to the LT limit]
In this paragraph we study the convergence of the mean delays to their LT limiting values, similarly to what we have done for the HT limit in the previous section. Although the LT limit of the mean delay \eqref{EWlt} has been obtained in a rather heuristical way, it turns out to be very accurate. In contrast to the HT limit, expression \eqref{EWlt} is not exact when the arrival processes are not Poisson, because the correction term $\rho_{g,j} \left(\E[\hat{A}_{g,j}]\hat{g}_{g,j}(0)-1\right) \E[B_{g,j}^\textit{res}])$ is based on a decomposition of the mean delay (see \cite{boonapprox2009} for more details) which is known not be true in general. Nevertheless, when comparing the approximated delays with simulated delays, results show that the approximation is very accurate (see Section \ref{numericalresults}). We have also been able to test the accuracy of the LT limit for small intersections (4 flows) by comparing it to \emph{exact} results under the assumption that all the departure headways, interarrival times and switch-over times are exponentially distributed. This has provided much insight in the behaviour of the delay as a function of the load $\rho$. Expression \eqref{EWlt} captures the LT behaviour very well, except for some cases where the behaviour exposes more curvature. To illustrate this, we consider a simple intersection with four flows, divided into two groups, each having a saturation flow rate of 1800 vehicles per hour. The mean all-red times are 6 seconds each. The relative loads are $\hat\rho_1=\hat\rho_3=0.1$, and $\hat\rho_2=\hat\rho_4=0.4$. We assume exponentially distributed headways, interarrival times and all-red times, because this enables us to model this system as a Markov chain. In order to solve this system numerically, we need a finite state space, so we take a maximum queue length of 6 vehicles per flow. This is sufficient since we are only focussing on LT behaviour. Figure \ref{accuracyLTpictures} shows the mean delays for flows 1 and 2 (and, because of symmetry, also of flows 3 and 4) as a function of $\rho$, the total load in the system. These pictures illustrate that the derivative at $\rho=0$ of the mean delay indeed equals the LT expression. But Figure \ref{accuracyLTpictures}(a) also illustrates that the actual behaviour may be non-linear, which can lead to deviations between the approximation, developed in the next section, and the actual mean delay.
\begin{figure}[h!]
\parbox{0.35\textwidth}{%
\begin{center}
\includegraphics[width=\linewidth]{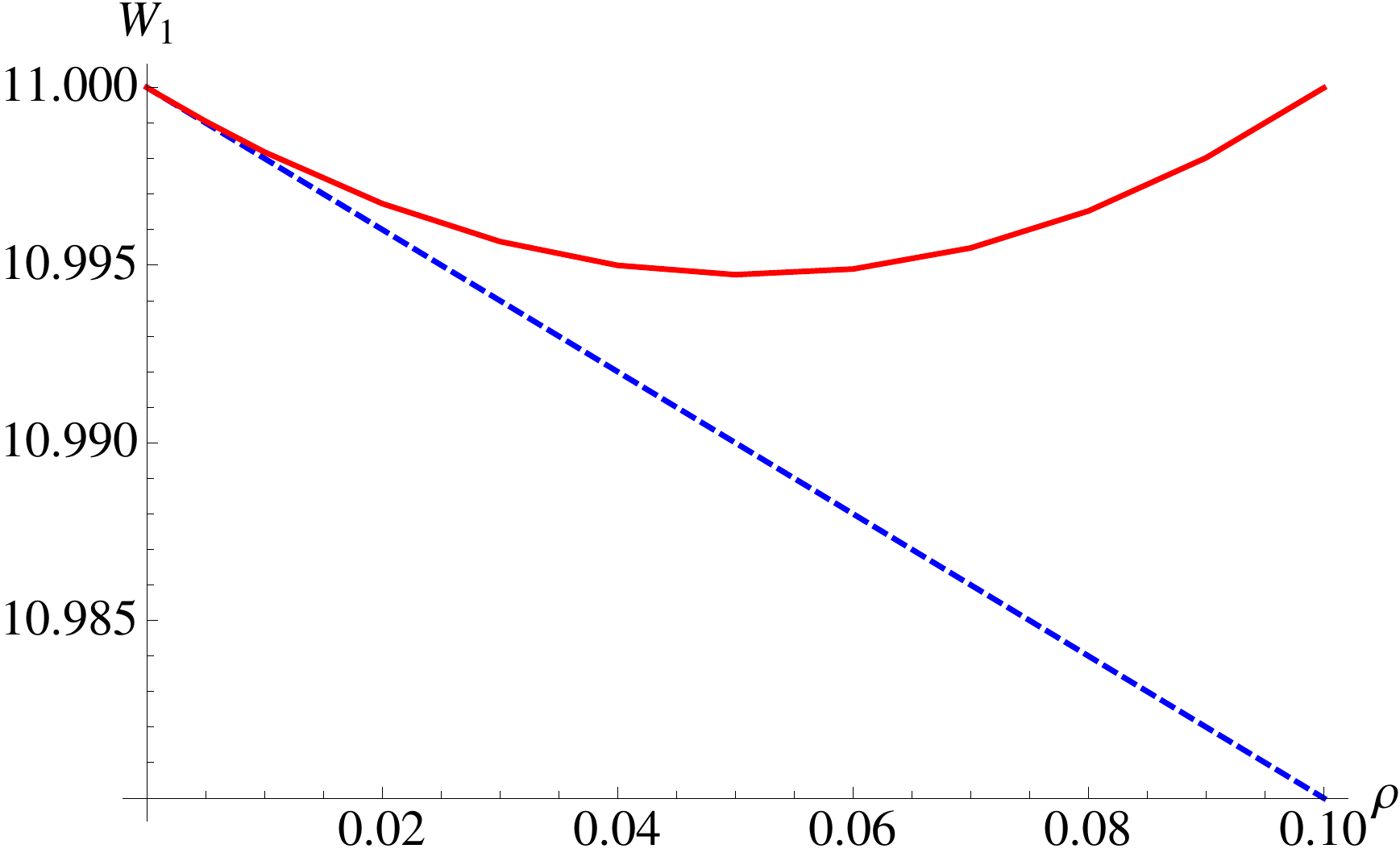}\\
(a)
\end{center}
}
\hfill
\includegraphics[width=0.2\linewidth]{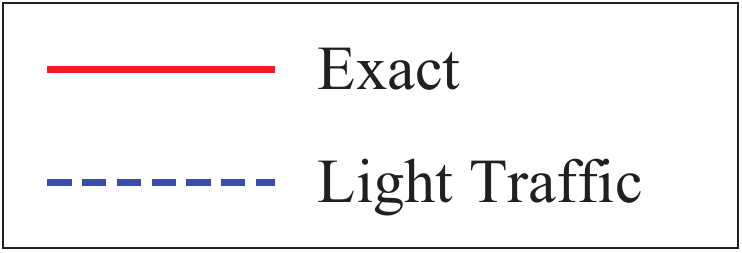}
\hfill
\parbox{0.35\textwidth}{%
\begin{center}
\includegraphics[width=\linewidth]{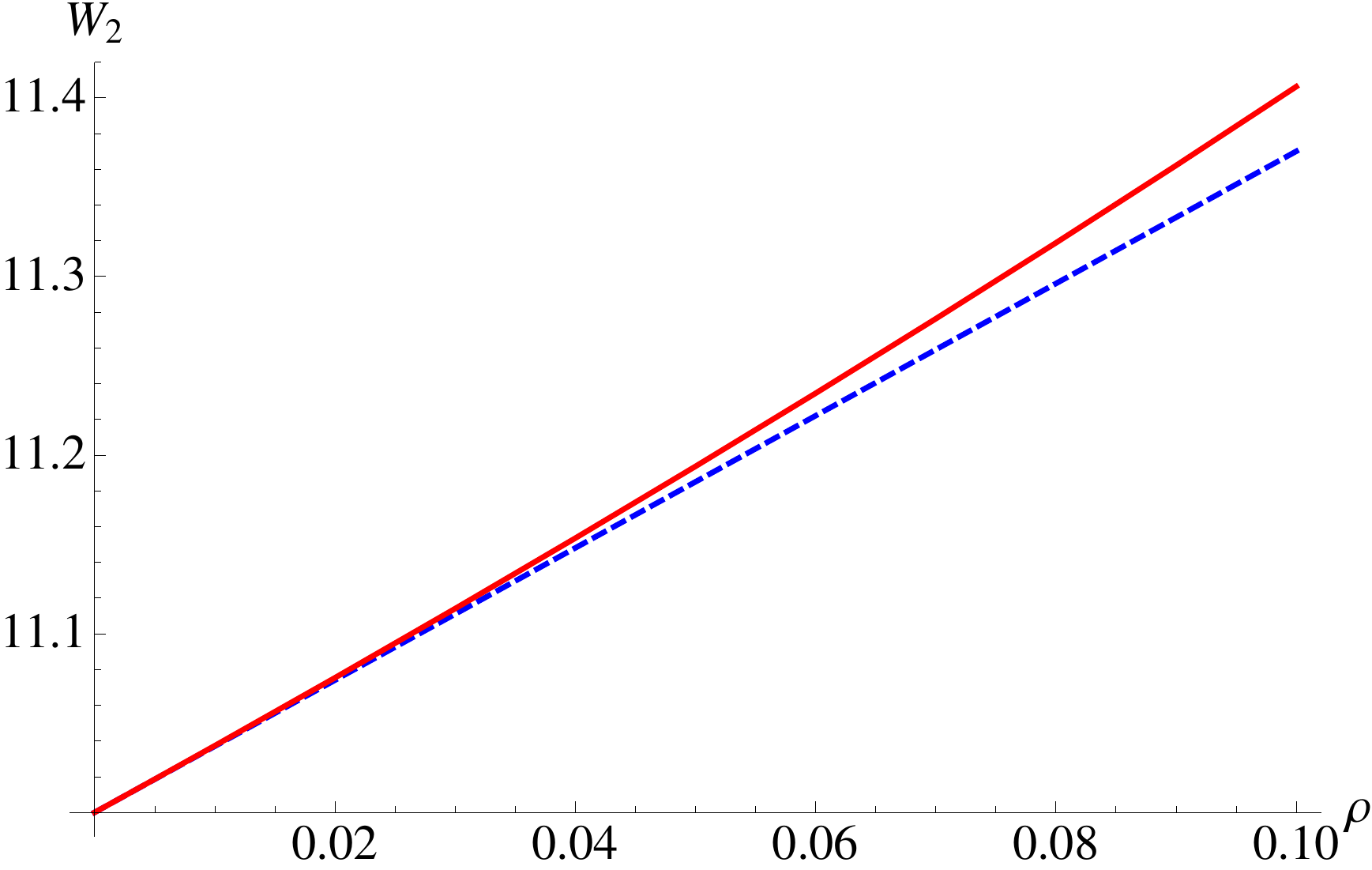}\\
(b)
\end{center}
}
\caption{The exact mean delays of flow 1 (left) and flow 2 (right) and the LT limits of the example in Section \ref{LTsection}.}
\label{accuracyLTpictures}
\end{figure}
\end{remark}

\section{Interpolations}\label{interpolationssection}

In the previous two sections we have given expressions for the mean delay in the extreme cases where the system is either hardly exposed to traffic, or completely saturated. In the present section we use these results to develop interpolations between these two cases that give an approximation for the mean delay for any load, as long as the system is not oversaturated (cf.~\cite{boonapprox2009}). In fact, we develop two different interpolations. At the end of this section we discuss the conditions under which each interpolation should be preferred.
\begin{align}
\E[W_{g,j}^\textit{app1}]&=\frac{K'_{0,g,j}+K'_{1,g,j} \rho}{1-L\rho},&\qquad g=1,\dots,M; j=1,\dots,N_g.\label{ewapprox1}\\
\E[W_{g,j}^\textit{app2}]&=\frac{K_{0,g,j}+K_{1,g,j} \rho + K_{2,g,j} \rho^2}{1-L\rho},&\qquad g=1,\dots,M; j=1,\dots,N_g.\label{ewapprox2}
\end{align}
De denominator $1-L\rho$ is required to capture the HT behaviour of the system, as discussed in Section~\ref{HTsection}. The numerator is a polynomial of first or second degree. The constants $K'_{0,g,j}$ and $K_{0,g,j}$ (which will turn out to be the same) follow from the requirement that \eqref{ewapprox1} and \eqref{ewapprox2} should result in the same mean delay for $\rho=0$ as the LT limit~\eqref{EWlt}.
We find $K'_{1,g,j}$ and $K_{2,g,j}$ by imposing that the HT limit of the approximation is equal to the HT limit of the exact mean delay.
Finally, the constant $K_{1,g,j}$ is found by adding the requirement that also the derivative with respect to $\rho$, taken at $\rho=0$, of our approximation is equal to the derivative of the LT limit.
A more formal definition of these requirements is presented below:

\begin{align*}
\E[W_{g,j}^\textit{app1}]\big|_{\rho=0} &= \E[W_{g,j}^\textit{LT}]\big|_{\rho=0},\\
\lim_{\rho\uparrow1/L}(1-L\rho)\E[W_{g,j}^\textit{app1}] &= \lim_{\rho\rightarrow1/L}(1-L\rho)\E[W_{g,j}],
\end{align*}%
and for approximation \eqref{ewapprox2},
\begin{align*}
\E[W_{g,j}^\textit{app2}]\big|_{\rho=0} &= \E[W_{g,j}^\textit{LT}]\big|_{\rho=0},\\
\frac{\dd}{\dd\rho}\E[W_{g,j}^\textit{app2}]\big|_{\rho=0} &= \frac{\dd}{\dd\rho}\E[W_{g,j}^\textit{LT}]\big|_{\rho=0},\\
\lim_{\rho\uparrow1/L}(1-L\rho)\E[W_{g,j}^\textit{app2}] &= \lim_{\rho\rightarrow1/L}(1-L\rho)\E[W_{g,j}],
\end{align*}
using \eqref{EWltDeterministic} for $\E[W_{g,j}^\textit{LT}]$ and \eqref{EWht} for $\lim_{\rho\rightarrow1/L}(1-L\rho)\E[W_{g,j}]$. This leads to the following constants:

\begin{align}
K'_{0,g,j} &= \frac{R}{2}+\E[B_{g,j}],    \label{Kacc0}\\
K'_{1,g,j} &= L\left(\frac{\left(1-\hat\rho_{g,1}/L\right)^2}{1-\hat\rho_{g,j}/L}  \left(\frac{R}{2} + \frac{\sigma^2}{\delta}  \right) - K'_{0,g,j}\right),\label{Kacc1}
\end{align}
\begin{align}
K_{0,g,j} &= \frac{R}{2}+\E[B_{g,j}],    \label{K0}\\
K_{1,g,j} &= \hat\rho_{g,j} \left(\E[\hat{A}_{g,j}]\hat{g}_{g,j}(0)-1\right) \E[B_{g,j}^\textit{res}]+\E[B^\textit{res}]-L\,\E[B_{g,j}]-\sum_{k\neq j} \hat\rho_{g,k}\left(\E[B^\textit{res}_{g,k}]+\E[B_{g,j}]\right)\nonumber\\
&+ \frac12(1-L+\hat\rho_{g,j}-2\hat\rho_{g,\bullet})R,     \label{K1}\\
K_{2,g,j} &= L^2\left(\frac{\left(1-\hat\rho_{g,1}/L\right)^2}{1-\hat\rho_{g,j}/L}  \left(\frac{R}{2} + \frac{\sigma^2}{\delta}  \right) - K_{0,g,j}\right) - L\,K_{1,g,j}, \label{K2}
\end{align}
where $\delta=\sum_{g=1}^M \frac{\hat\rho_{g,1}}{L}(1-\frac{\hat\rho_{g,1}}{L})/2$, and $\sigma^2=\sum_{g=1}^M \hat\lambda_{g,1}\left(\V[B_{g,1}]+\hat\rho_{g,1}^2\V[\hat A_{g,1}]\right)$.
In the above expressions, we have assumed that the all-red times are deterministic, as is usually the case. Sections~\ref{HTsection} and~\ref{LTsection} give insight in how to adapt these constants if randomness is involved in (some of) the all-red times.

An obvious question that arises now, is which of the two approximations, \eqref{ewapprox1} or \eqref{ewapprox2}, should be preferred.
We can give some (heuristic) arguments, obtained by studying many numerical examples.
It can be shown that the two approximations, \eqref{ewapprox1} or \eqref{ewapprox2}, are both exact for the limiting cases $\rho\rightarrow 0$, ${\rho\rightarrow1/L}$ and $\E[R]\rightarrow\infty$. The difference between the two interpolation functions, respectively having a first and second order polynomial in the numerator, is the additional requirement that the derivative  in $\rho=0$ of the approximated mean waiting time should be equal to the derivative of the LT expression \eqref{EWlt}. Since the LT expression is very accurate, using this additional information generally leads to more accurate approximations. For this reason, \eqref{ewapprox2} should be preferred in most practical cases. However, there are some circumstances under which \eqref{ewapprox1} should be preferred, despite resulting in a derivative at $\rho=0$ which is not exact. In more detail, if the actual delay displays a strong non-linear behaviour for small loads, using the second-order interpolation leads to bigger inaccuracies for loads around $L\rho=0.7$. In most cases, this happens if the derivative in $\rho=0$ is very small, or even negative (see also Figure \ref{LTsection}(a)). Studying the LT expression \eqref{EWlt} shows that the most natural way for a negative derivative to appear, is when the combined load in all other groups is smaller than the load of the other flows within the group. In terms of the model parameters: if, for a certain flow $\{g,i\}$, the criterion
\begin{equation}
\sum_{m\neq g}\hat\rho_{m,\bullet} - \sum_{j\neq i}\hat\rho_{g,j}<0,\label{firstordercriterion}
\end{equation}
is satisfied, the first-order interpolation \eqref{ewapprox1} is preferred over the second-order interpolation \eqref{ewapprox2}. Note that this is just a rule of thumb. We show the effect of choosing the wrong interpolation at the end of Example 2 in the next section.

\begin{remark}
A negative slope at $\rho=0$ is possible because of the assumption that the delay of vehicles approaching an empty flow during a green period experience no delay at all. Under some circumstances, mostly when Condition \ref{firstordercriterion} is satisfied, an increase in traffic may be beneficial for certain flows that hardly receive any traffic at all. The increase in traffic results in an increase of the mean green period, which results in a larger number of vehicles that benefit from the green light while their flow is empty. This might have a positive effect on the mean delay for vehicles in this flow, although the effect disappears when the load increases further.
\end{remark}

\section{Numerical results}\label{numericalresults}

In the present section, we study typical features of the approximations and assess their accuracies. In Example~1 we do this by taking an imaginary intersection and simulating several scenarios. For each scenario we compare the simulated delays to the approximated delays. In Example~2 we take three real, existing intersections.

\subsection*{Example 1: Accuracy of the approximations}

In this example, we analyse an (imaginary) intersection with 6 traffic flows. The ratios of the arrival rates of the six flows are $1:2:3:4:5:6$. The discharge rates are all equal, 1800 vehicles per hour, hence $\E[B_i]=2$ seconds.
For now, we assume that the SCV of the departure headways is 1. In our simulation we have used exponentially distributed random variables to achieve this. We will have a short discussion at the end of this example about other probability distributions. Note that $\hat\rho_i=\frac{i}{21}$, because the mean headways are all equal. The total all-red time $R$ in a cycle is 12 seconds, divided equally among the individual all-red times. We compare several different scenarios to get insight into the accuracy of the approximations. The first seven scenarios, shown in Table \ref{scenariosExample1}, differ in the choice of which flows to combine in one group.
For each scenario, the mean delays for all queues have been obtained by simulation, and are being compared to the approximated mean delays \eqref{ewapprox1}  or \eqref{ewapprox2}, using Criterion \eqref{firstordercriterion} to decide which of the two interpolations is used. The mean delay of each queue, in each scenario, is obtained for 11 different loads: $L\rho = \{0.001, 0.1, 0.2, \dots, 0.8, 0.9, 0.99\}$. Note that it is convenient to use $L\rho$ instead of $\rho$, because $L\rho$ takes values between 0 and 1. Furthermore, $L\rho$ can be interpreted as the level of saturation of the intersection.
By comparing all approximated mean delays to the corresponding simulated values, the relative errors are calculated:
\[
e_i(L\rho)=\left|\frac{W_i^\textit{app}-W_i^\textit{sim}}{W_i^\textit{sim}}\right|\,\times 100\%, \qquad i=1,\dots,6.
\]
In order to compare the scenarios, we introduce two quality measures. The first quality measure, \textit{QM1}, which gives insight in the \emph{worst} performance, is the \emph{largest} relative error (and the corresponding queue, and the load at which this error is obtained).
The second quality measure, \textit{QM2}, which provides a better insight in the \emph{overall} performance, is the weighted mean relative error. This is computed by averaging the 11 relative errors for each flow, $\overline{e}_i$, and subsequently taking the weighted average proportional to the arrival rates. In more detail:
\begin{align*}
\textit{QM1} &= \max_{i,L\rho}e_i(L\rho),\\
\textit{QM2} &= \sum_{i=1}^6 \frac{\hat\lambda_i}{\sum_{j=1}^6\hat\lambda_j}\overline{e}_j.\\
\end{align*}
Table \ref{scenariosExample1} displays these two quality measures for scenarios \I--\VII.
\begin{table}[h!]
\begin{center}
\begin{tabular}{|c|c|c|c|c|c|c|c|c|}
\hline
Scenario & Groups & $\C_{A_i}^2$ & $\C_B^2$ & Interpolation & \multicolumn{3}{|c|}{\textit{QM1}} & \textit{QM2}\\
\cline{6-8}
         &        &        &            & orders& error & flow & $L\rho$ & \\
\hline
\I     & $\{1\}, \{2\}, \{3\}, \{4\}, \{5\}, \{6\}$ & 1 & 1 & 2 2 2 2 2 2 & 0.3\% & 1 & 0.7 & 0.06\%\\
\hline
\II    & $\{1,2\},\{3,4\},\{5,6\}$ & 1 & 1 & 2 2 2 2 2 2 & 21.9\% & 6 & 0.9  & 8.17\% \\
\hline
\III   & $\{1,4\},\{2,5\},\{3,6\}$ & 1 & 1 & 2 2 2 2 2 2 & 4.4\%  &6  & 0.7  & 1.29\%\\
\hline
\IV    & $\{1,6\},\{2,5\},\{3,4\}$ & 1 & 1 & 2 2 2 2 2 2 & 10.3\% & 5 & 0.9 & 3.29\%\\
\hline
\VV     & $\{1,2,3\},\{4,5,6\}$    & 1 & 1 & 2 2 2 1 1 1 & 12.3\% & 6 & 0.9 & 4.14\%\\
\hline
\VI    & $\{1,2,5\},\{3,4,6\}$    & 1 & 1 & 2 2 1 1 2 2 & 11.8\% & 6 & 0.7 & 3.79\%\\
\hline
\VII   & $\{1,3,5\},\{2,4,6\}$    & 1 & 1 & 2 1 2 2 2 2 & 9.5\% & 6 & 0.7 & 3.22\% \\
\hline
\end{tabular}
\end{center}
\caption{Scenarios \I--\VII\ used in Example 1, and the corresponding quality measures \textit{QM1} and \textit{QM2}.}
\label{scenariosExample1}
\end{table}

We discuss the scenarios \I--\VII\ first, later we add some extra scenarios. Scenario \I\ divides the six signals into six separate groups. This reduces the model to an ordinary polling model with exhaustive service and Poisson arrivals. Exact results are known for this model, and the approximation developed in the present paper simplifies to the expression found in \cite{boonapprox2009}. In \cite{boonapprox2009}, this approximation is discovered to perform very well, especially for systems with more than two queues. Therefore, it is no surprise that scenario \I\ results in the best approximation accuracy: an average mean error of only 0.06\%, and the worst relative error, 0.3\%, is obtained for flow 1 (with the smallest load) at $L\rho = 0.7$. The next scenario, Scenario \II, has the lowest accuracy, illustrating the greatest drawbacks of the approximation. The six flows are divided into three groups. The reason why this scenario performs so poorly, is that the flows in each group have rather similar loads. When the HT limit is reached, the system behaviour is such that the heaviest loaded flows dominate the groups and determine the lengths of the green periods. As a result, the approximation converges to the HT limit only very slowly. For $L\rho=0.9$ the relative errors $e_i$ are approximately 21\% for \emph{all} flows, whereas the errors are back to 7\% for $L\rho=0.99$. The overall mean relative error of 8.17\% is reasonable, but in the range $0.8 \leq L\rho \leq 0.97$ all relative errors are greater than 15\%. We will have some further discussion on this issue later in this paper. In Scenario \III, the flows are also divided into three groups, but now the best possible choice (for our approximation) is made, avoiding groups with two flows having similar loads. Therefore, the results are excellent now, with $\textit{QM1}=4.4\%$ and $\textit{QM2}=1.29\%$. Scenario \IV\ might be the most interesting from a practical point of view, combining good and bad combinations of flows in the groups. As expected, the performance is better than Scenario \II, and worse than Scenario \III, with $\textit{QM1}=10.3\%$ and $\textit{QM2}=3.29\%$. Because there is one group (with flows 3 and 4) having loads relatively close to each other, the relative errors are relatively large for $L\rho=0.9$, ranging from $7.5\%-10\%$, but since the other groups have better combinations of flows, the results are much better than for Scenario \II. We will return to Scenario~\IV\ later in this section, when we vary several other parameters of the model.

\begin{figure}[h!]
\begin{center}
\mbox{}\hfill
\includegraphics[width=0.45\linewidth]{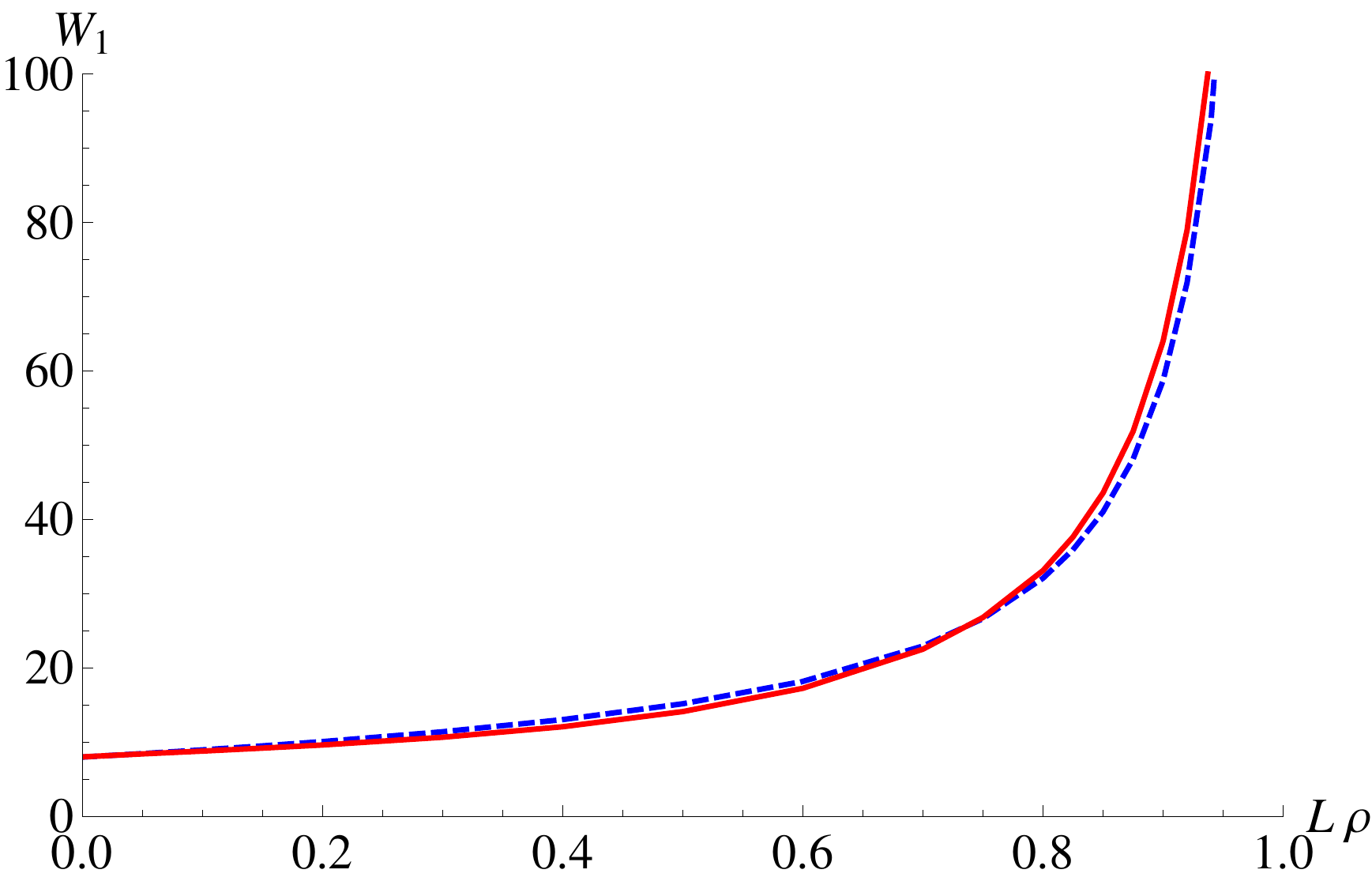}
\hfill
\includegraphics[width=0.45\linewidth]{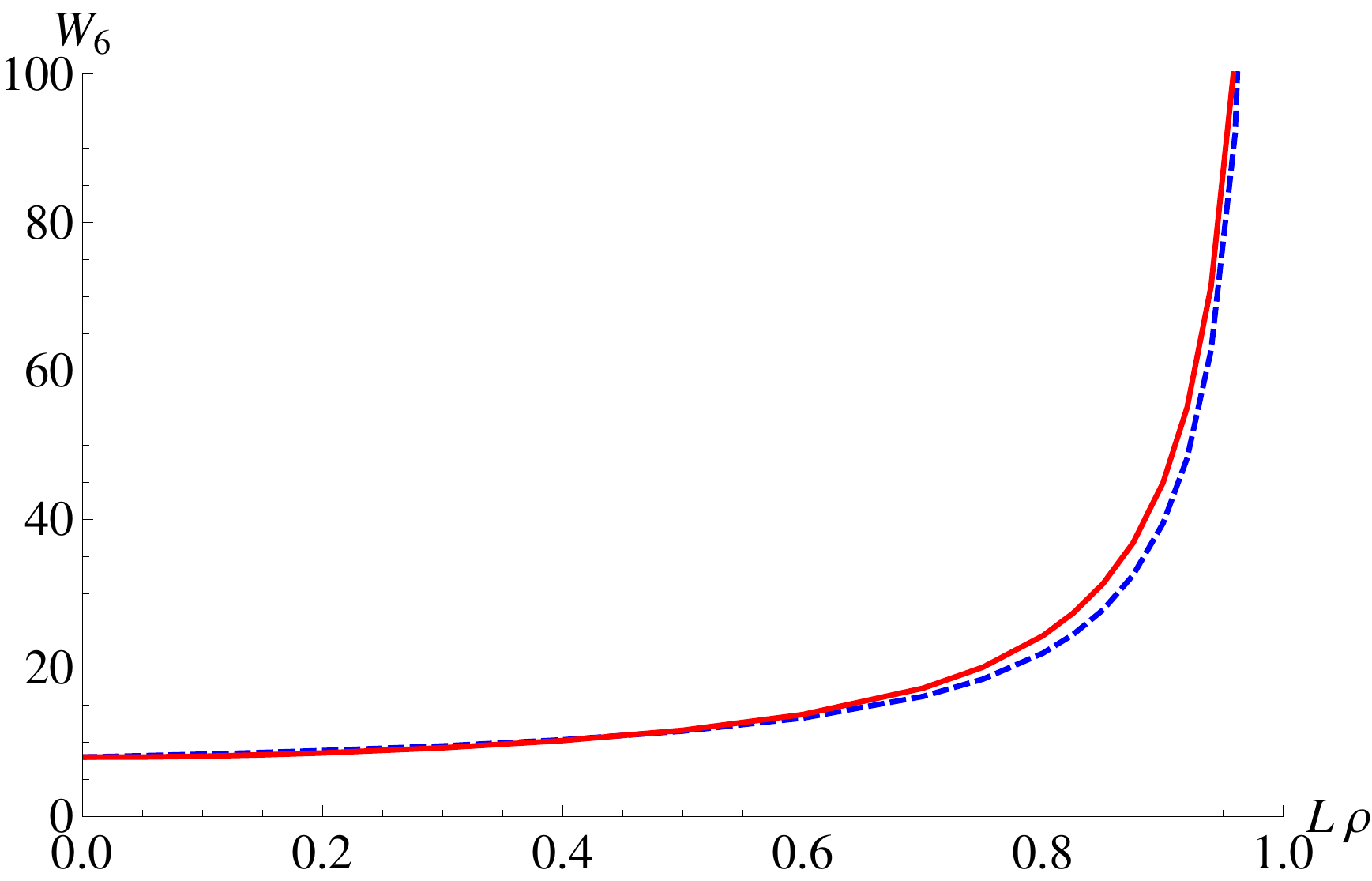}
\hfill
\mbox{}\\
\mbox{}\hfill
\includegraphics[width=0.2\linewidth]{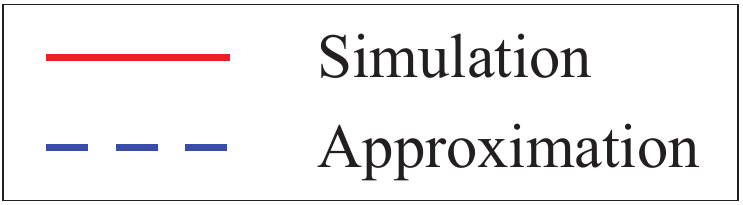}
\hfill\mbox{}
\caption{Simulated and approximated mean delays for flow 1 (left) and flow 6 (right) in Scenario V.}
\label{scenario5}
\end{center}
\end{figure}

Now, we discuss Scenarios \VV-\VII\ briefly. These scenarios have only two groups, each containing three flows. Scenario \VV\ can be compared to Scenario \II, having the worst possible division into groups. An interesting aspect of this scenario, is that the mean delays of all flows in group 1 are approximated using the second-order approximation \eqref{ewapprox2}, whereas the mean delays in group 2 use the first-order approximation \eqref{ewapprox1}. The reason to do this, is that the Criterion \eqref{firstordercriterion} is satisfied for the flows in group 2, because this group contains the three heaviest loaded flows of the intersection. This implies that the total relative load of group 1, $\hat\rho_{1,\bullet}$, is much smaller than the relative load of any pair of flows in group 2, $\sum_{j\neq i}\hat\rho_{2,j}$, for $i=1,2,3$. Since Criterion \eqref{firstordercriterion} is satisfied, second order interpolations would severely underestimate the mean delays in group 2. In fact, if one would take \eqref{ewapprox2} instead of \eqref{ewapprox1} for group 2, the mean relative error for flows in this group would be increased from $4.07\%$ to $15.9\%$. Figure \ref{scenario5} shows the approximated and simulated mean delays for flow 1 (typical for group 1) and for flow 6 (typical for group 2) against $L\rho$. It can be seen that the derivative in $L\rho=0$ of the approximated mean delay for flow 1 is not correct (it is slightly greater than the actual value), but it does lead to an overall better accuracy. In Scenarios \VI\ and \VII, the division into groups is better, resulting in more favourable results.

Scenarios $\I-\VII$ illustrate how the accuracy of the approximation depends on the distribution of the loads among the groups. Numerical experiments indicate that this is the biggest source of variation in the accuracy. For completeness, we show the effects of different interarrival-time distributions, and different distributions of the departure headways in Scenario \textit{IV}. In our simulations we have fitted phase-type distributions, as suggested in \cite{tijms94}, matching the specified SCVs. We can be brief in discussing the results, displayed in Table \ref{scenarios2Example1}. The variation in the interarrival times and headways does not have a major impact on the accuracy of the approximation, but in general it can be concluded that an increase (decrease) in variation results in a decrease (increase) in the accuracy.
\begin{table}[h!]
\begin{center}
\begin{tabular}{|c|c|c|c|c|c|c|c|}
\hline
Scenario & Groups & $\C_{A_i}^2$ & $\C_B^2$ & \multicolumn{3}{|c|}{\textit{QM1}} & \textit{QM2}\\
\cline{5-7}
         &        &        &            & error & flow & $L\rho$ & \\
\hline
\IV    & $\{1,6\},\{2,5\},\{3,4\}$ & 1 & 1 & 10.3\% & 5 & 0.9 & 3.29\%\\
\hline
\hline
\VIII    & $\{1,6\},\{2,5\},\{3,4\}$ & 0.5 & 1 & 7.8\% & 5 & 0.9 & 1.91 \%\\
\hline
\IX    & $\{1,6\},\{2,5\},\{3,4\}$ & 2 & 1 & 14.7\% & 5 & 0.9 & 5.70\%\\
\hline
\X    & $\{1,6\},\{2,5\},\{3,4\}$ & 1 & 0 & 5.6\% & 4 & 0.9 & 1.57\%\\
\hline
\XI    & $\{1,6\},\{2,5\},\{3,4\}$ & 1 & 0.5 & 8.2\% &  5 & 0.9 & 2.50\%\\
\hline
\XII    & $\{1,6\},\{2,5\},\{3,4\}$ & 1 & 2 & 13.1\% & 5 & 0.9 & 4.45\%\\
\hline
\end{tabular}
\end{center}
\caption{Scenarios \IV, \VIII--\XII\ used in Example 1, and the corresponding quality measures \textit{QM1} and \textit{QM2}.}
\label{scenarios2Example1}
\end{table}


\subsection*{Example 2: Real-life examples}

\newlength{\figwidth}
\setlength{\figwidth}{0.3\textwidth}
\begin{figure}[htb]
\begin{center}
\parbox{\figwidth}{%
\begin{center}
\includegraphics[width=\linewidth]{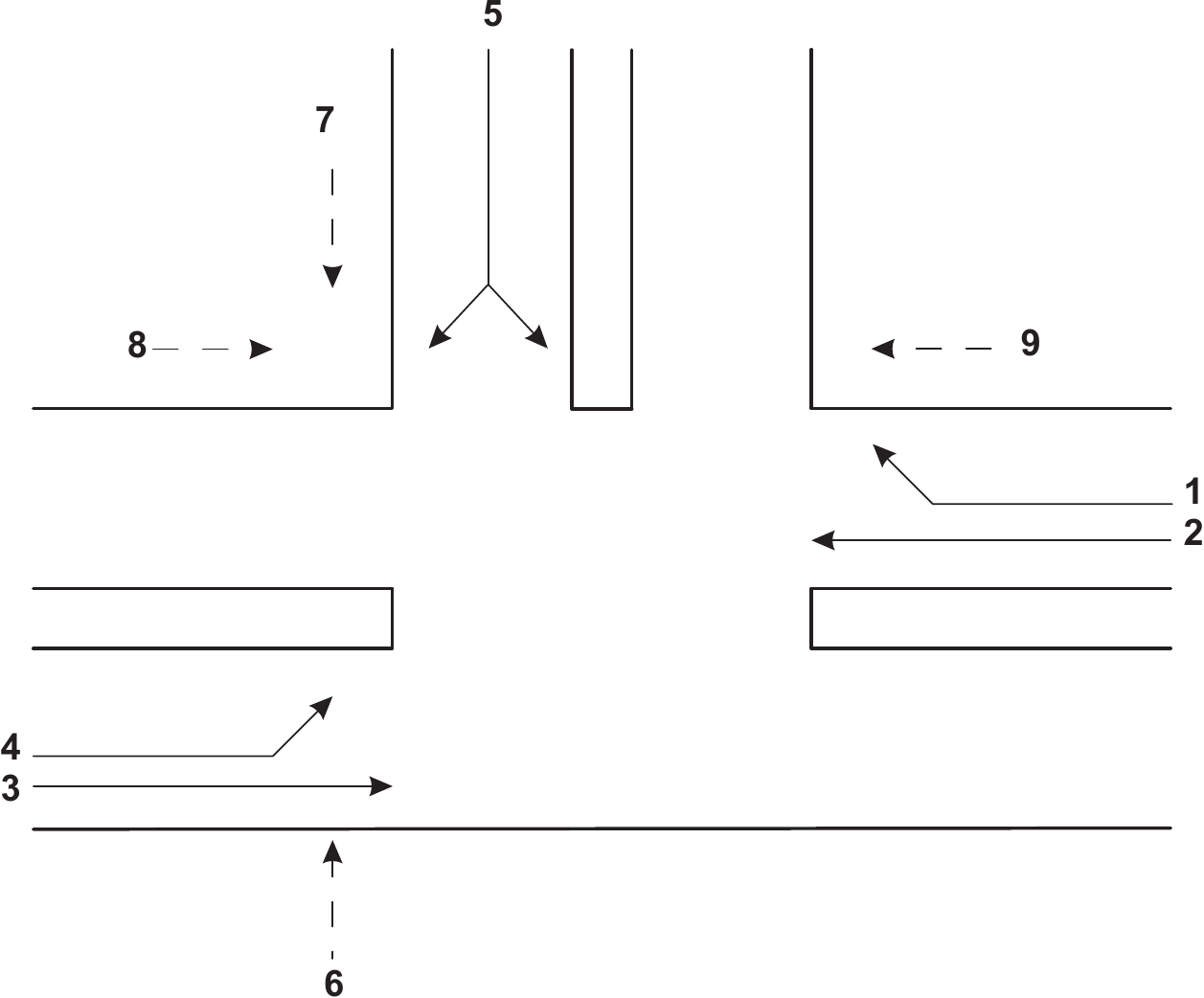}
\end{center}
}
\hfill
\parbox{\figwidth}{%
\begin{center}
\includegraphics[width=\linewidth]{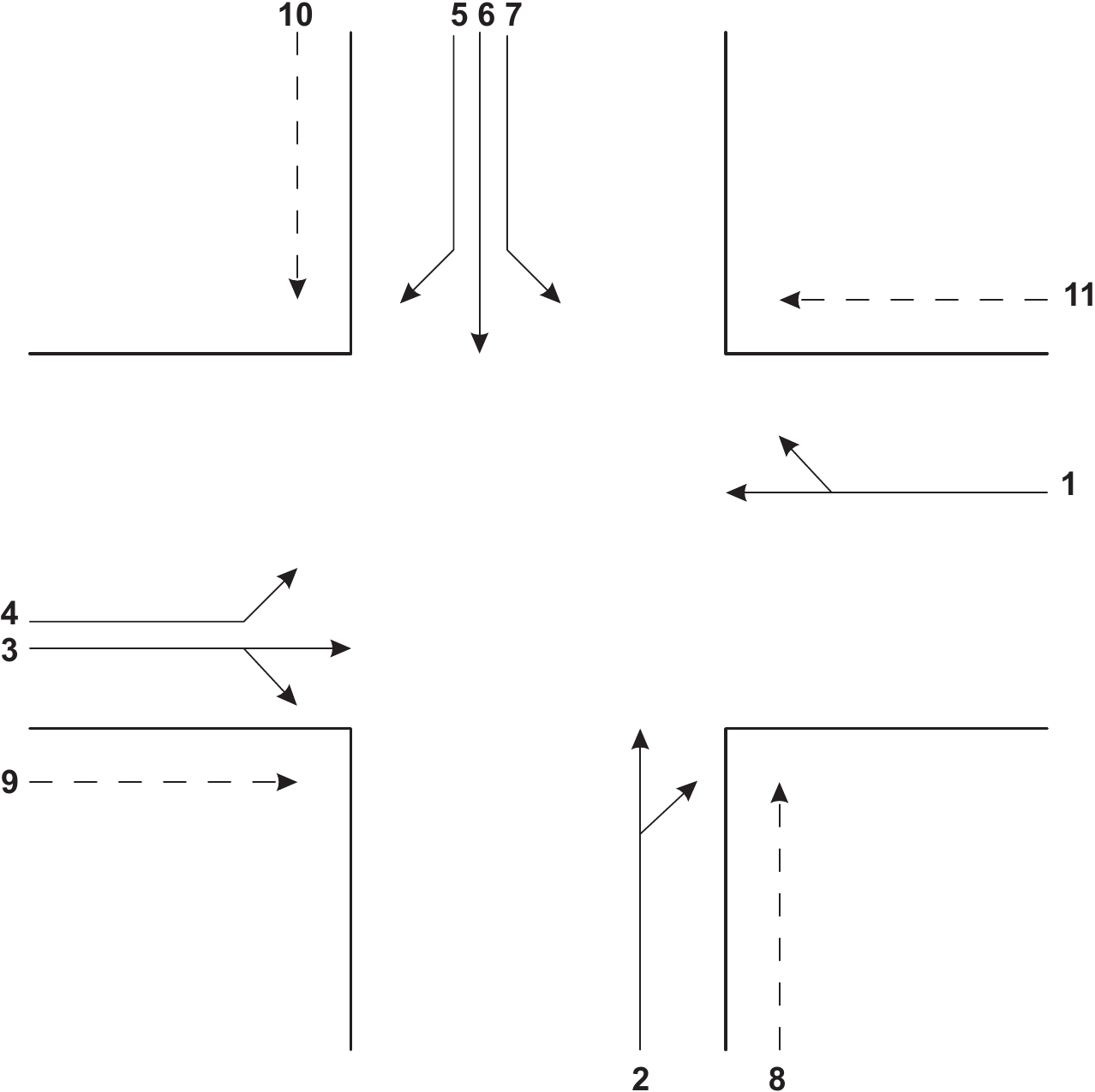}
\end{center}
}
\hfill
\parbox{\figwidth}{%
\begin{center}
\includegraphics[width=\linewidth]{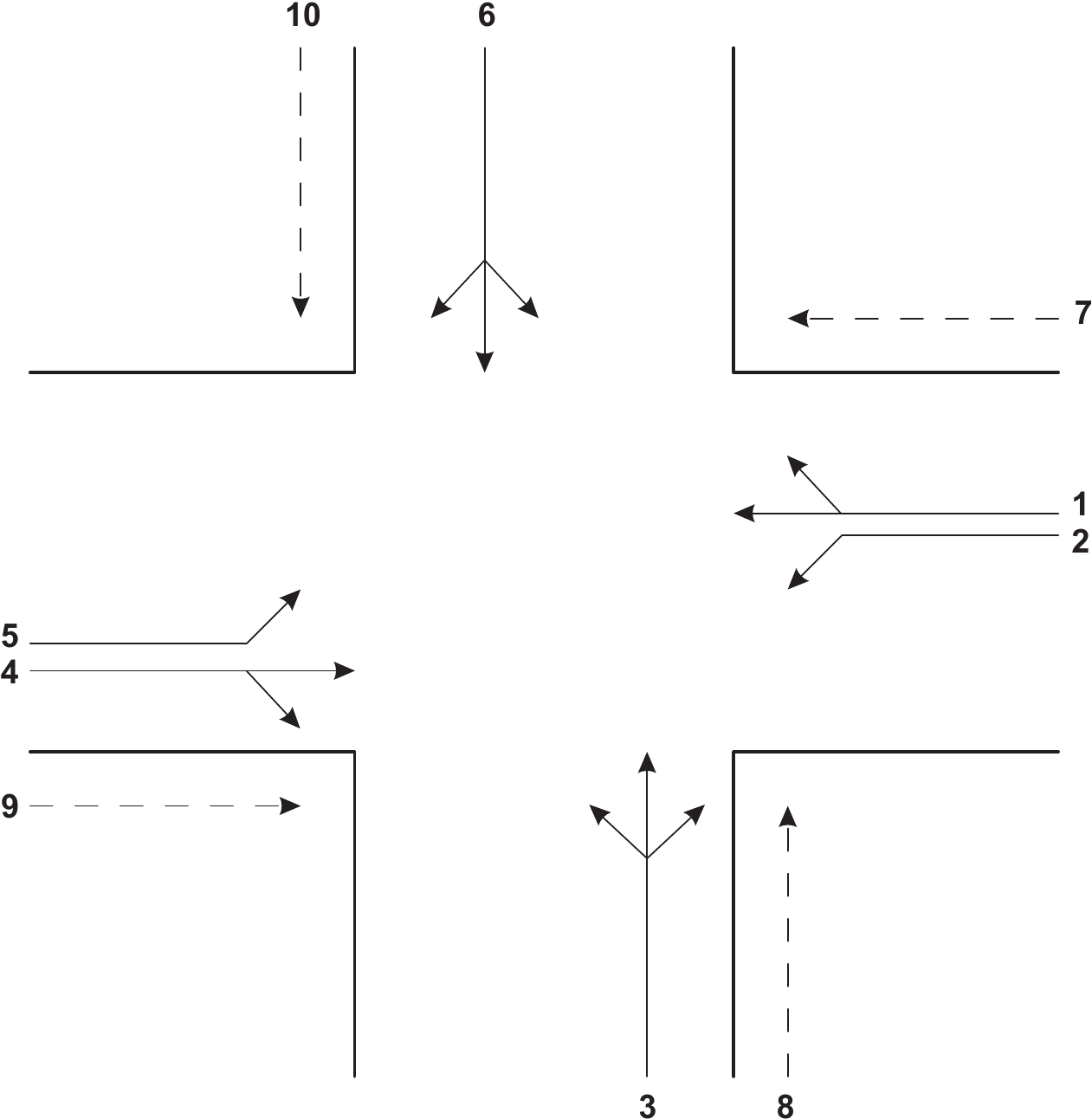}
\end{center}
}\\
\parbox{\figwidth}{%
\begin{center}
Intersection 1
\end{center}
}
\hfill
\parbox{\figwidth}{%
\begin{center}
Intersection 2
\end{center}
}
\hfill
\parbox{\figwidth}{%
\begin{center}
Intersection 3
\end{center}
}\\
\fbox{\includegraphics[width=0.3\textwidth]{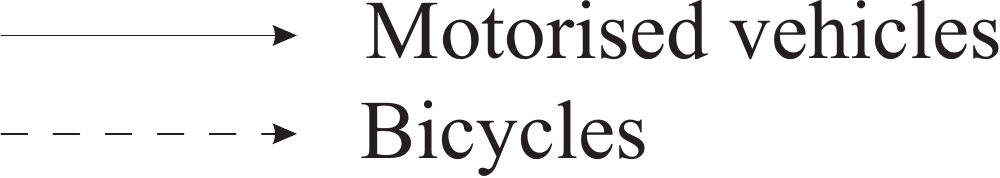}}
\end{center}
\caption{The intersections discussed in Example 2.}
\label{example2intersections}
\end{figure}
In the present subsection we test the approximation on three real-life situations. We take three intersections, graphically displayed in Figure \ref{example2intersections}, and compare approximated mean delays with the simulated values. The first two intersections are located in Eindhoven, The Netherlands, and data have been obtained from the local city council. The data for the third intersection are taken from the Dutch manual for configuration of traffic intersections \cite{handboekcrow}.
Each of these intersections contains several flows for motorised vehicles, and four bicycle lanes. The exact settings for each intersection can be found in Appendix \ref{appendixintersections}. Table \ref{intersectionsExample2} shows the numerical results of the approximated and simulated delays for the three different intersections.
\begin{table}[h!]
\begin{center}
\begin{tabular}{|c|c|c|c|c|}
\hline
Intersection & \multicolumn{3}{|c|}{\textit{QM1}} & \textit{QM2}\\
\cline{2-4}
     & error & flow & $L\rho$ & \\
\hline
1    &  21.3\% & 2 & 0.99 & 6.60\%\\
\hline
2    &  13.6\% & 6 & 0.9 & 4.65 \%\\
\hline
3    &  30.4\% & 4 & 0.9 & 11.62\%\\
\hline
\end{tabular}
\end{center}
\caption{Performance measures \textit{QM1} and \textit{QM2} for the intersections in Section \ref{numericalresults}, Example 2.}
\label{intersectionsExample2}
\end{table}
For all of the intersections, but especially for Intersections 1 and 2, the overall accuracy is quite good, considering the overall mean relative error percentage \textit{QM2}. But at the same time, one of the drawbacks becomes apparent. In practice, these intersections have maximum green times, which are chosen such that the cycle time cannot exceed a specified value. In order to minimise the maximum cycle time, flows are generally divided into groups such that the busiest streams are put together in the same group, as long as no conflicts arise. As discussed in the first example, our approximation gives very accurate results, as long the loads of the busiest flows in a group are not too close to each other. However, some of these examples have at least one group with two busy flows with similar relative loads. In Intersection 1, in group 4, we have $\hat\rho_{4,1}=0.12$ and $\hat\rho_{4,2}=0.11$. In Intersection 3, the highest two relative loads in group 3 are $\hat\rho_{3,1}=0.15$ and $\hat\rho_{3,2}=0.14$. This means that the mean waiting time converges to its HT limit \eqref{EWht} only very slowly. This explains that for these two intersections, the approximation gives rather high relative errors (sometimes more than 20\%) for high loads $(L\rho=0.9,\dots,0.98)$.

\begin{figure}[h!]
\setlength{\figwidth}{0.42\textwidth}
\mbox{}\hfill\parbox{\figwidth}{%
\begin{center}
\includegraphics[width=\linewidth]{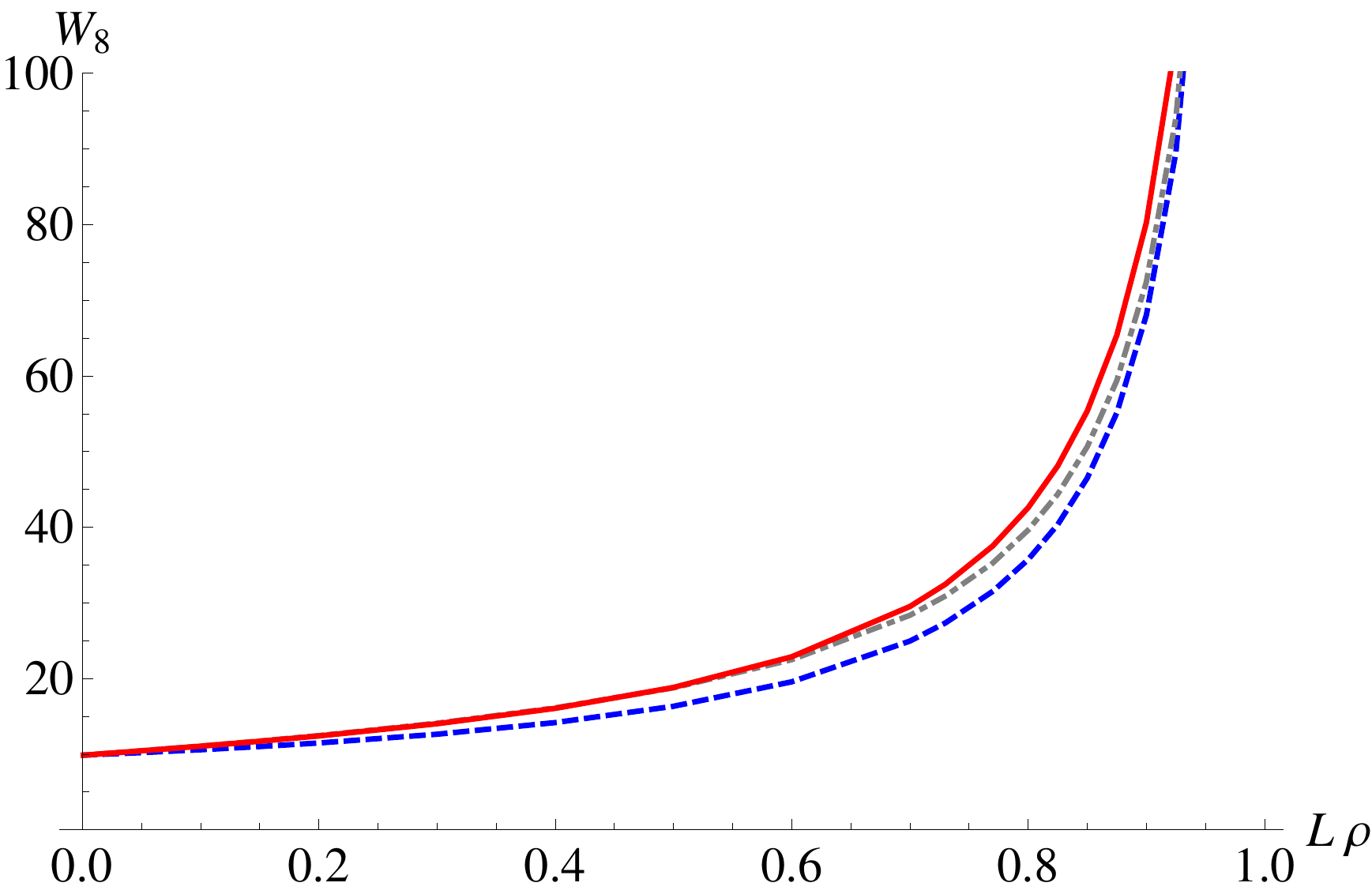}\\
(a)
\end{center}
}
\hfill
\parbox{\figwidth}{%
\begin{center}
\includegraphics[width=\linewidth]{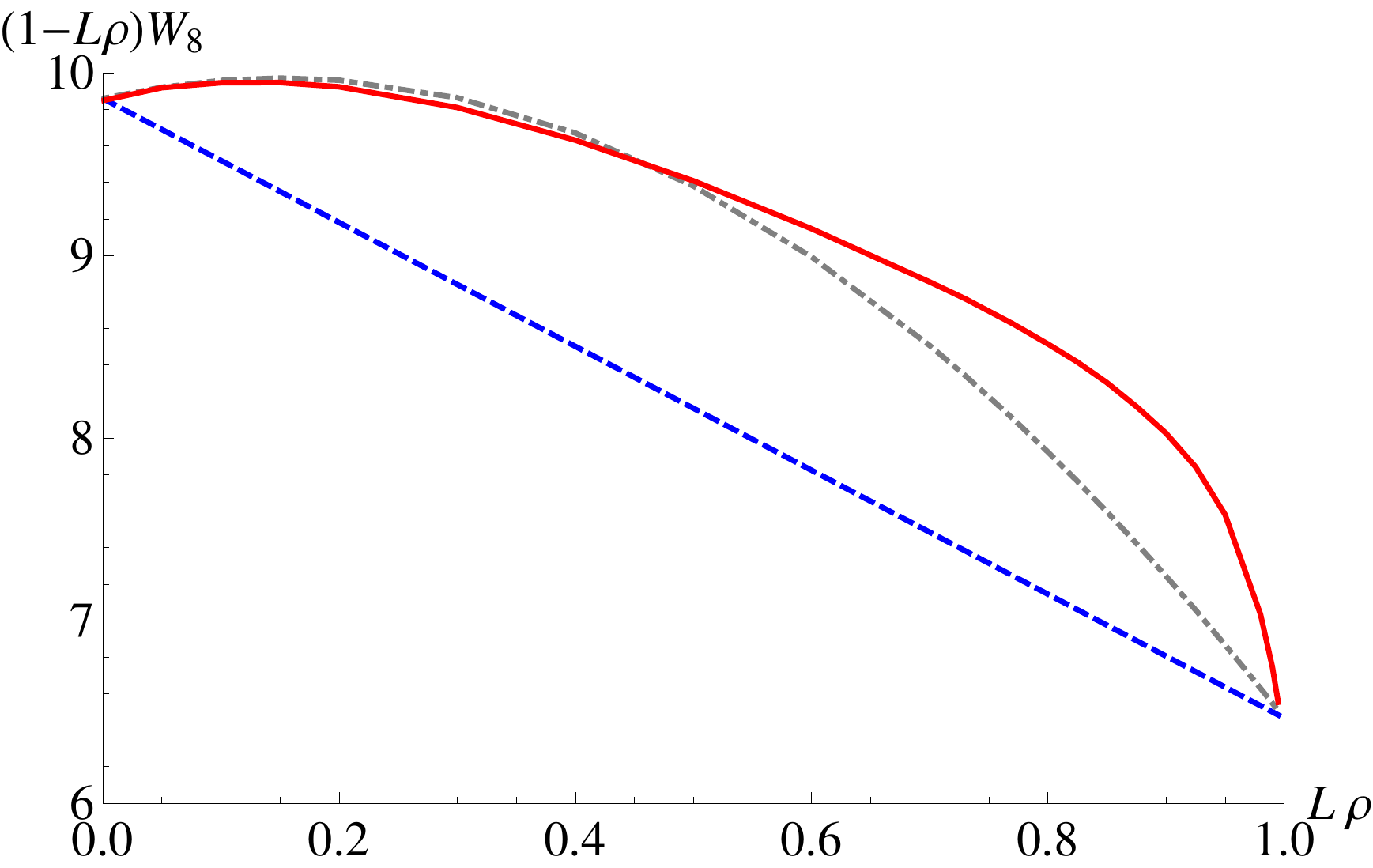}\\
(b)
\end{center}
}
\hfill
\mbox{}\\
\mbox{}\hfill
\parbox{\figwidth}{%
\begin{center}
\includegraphics[width=\linewidth]{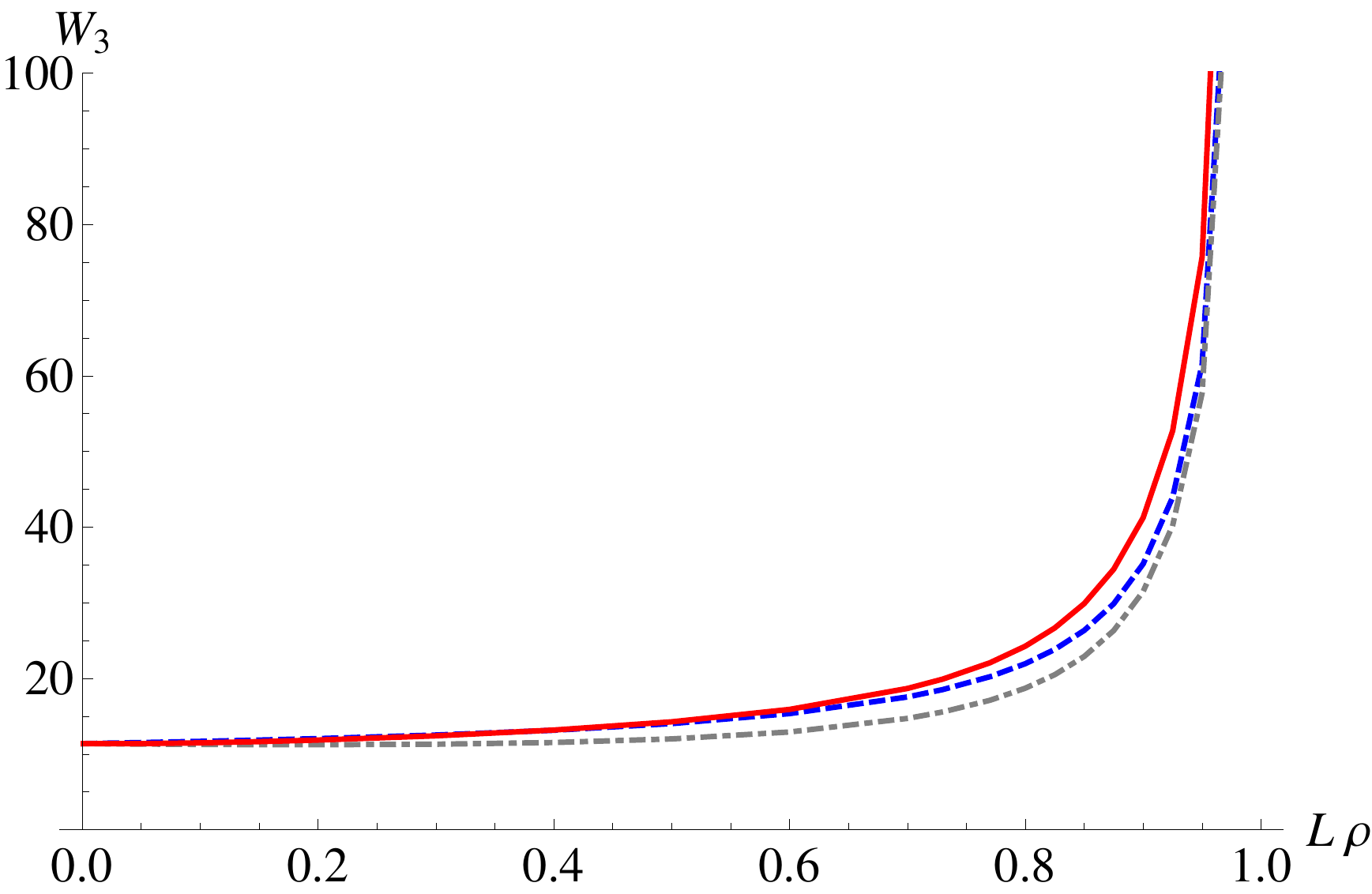}\\
(c)
\end{center}
}
\hfill
\parbox{\figwidth}{%
\begin{center}
\includegraphics[width=\linewidth]{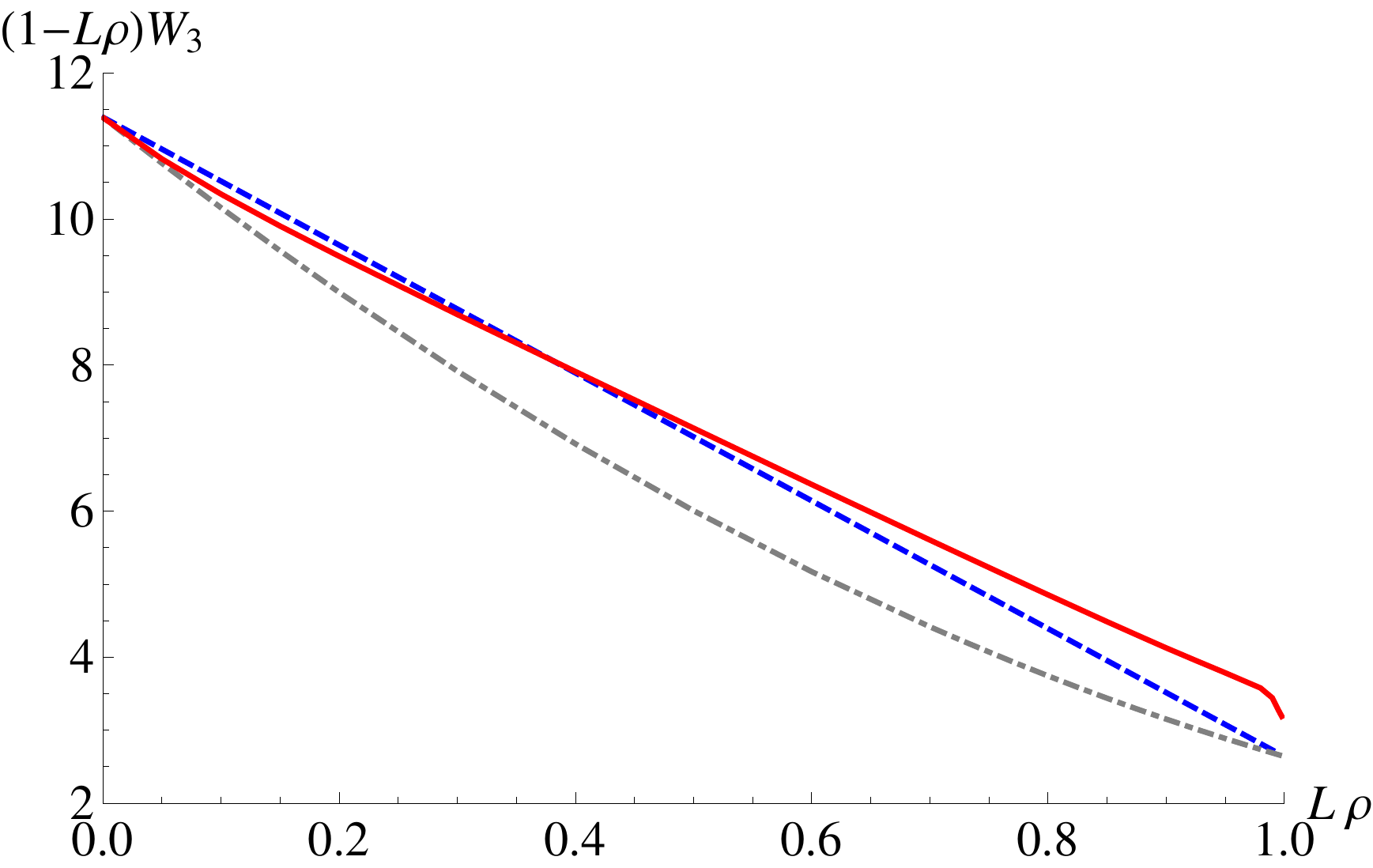}\\
(d)
\end{center}
}
\hfill
\mbox{}\\
\mbox{}\hfill
\parbox{\figwidth}{%
\begin{center}
\includegraphics[width=\linewidth]{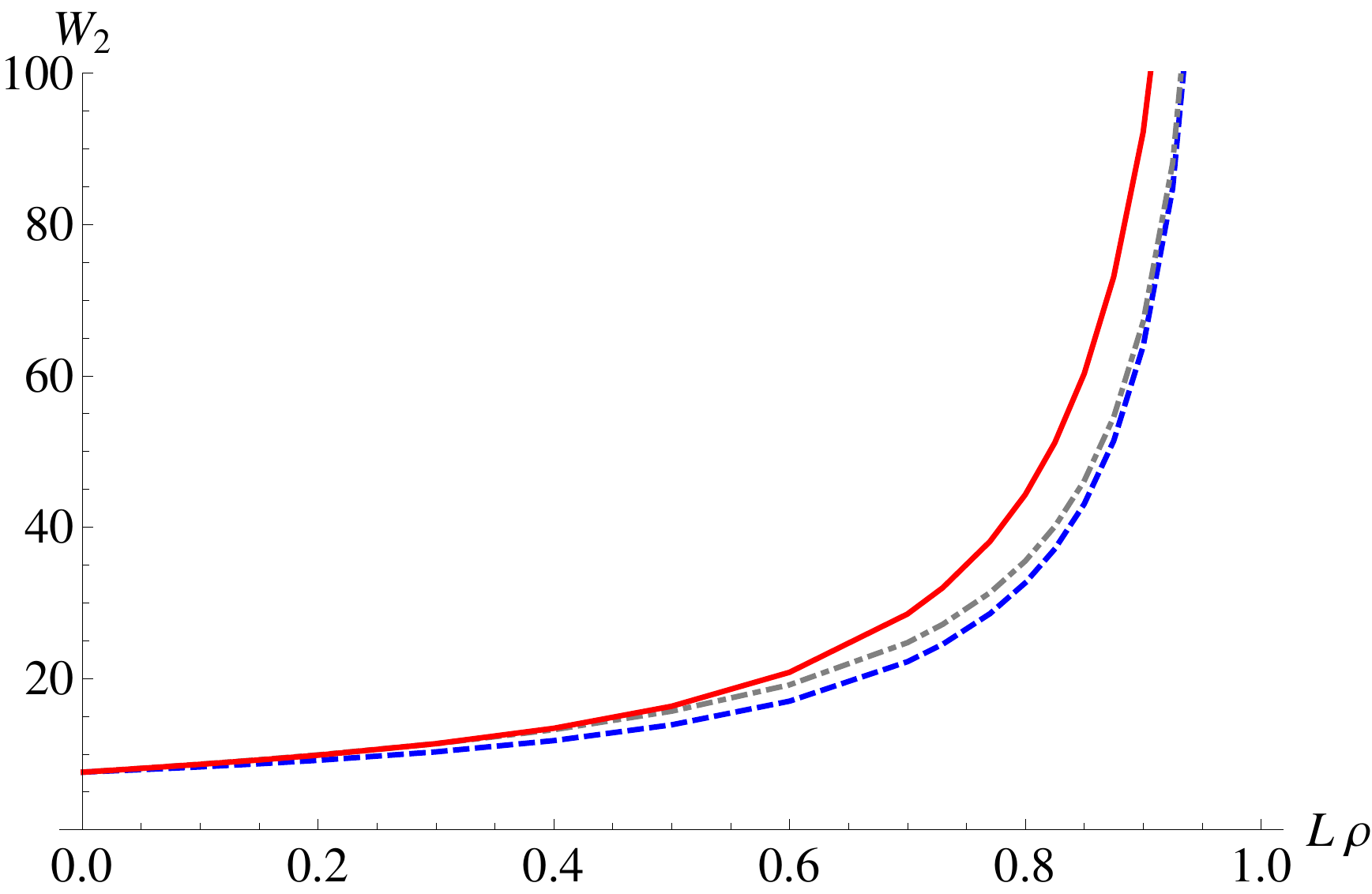}\\
(e)
\end{center}
}
\hfill
\parbox{\figwidth}{%
\begin{center}
\includegraphics[width=\linewidth]{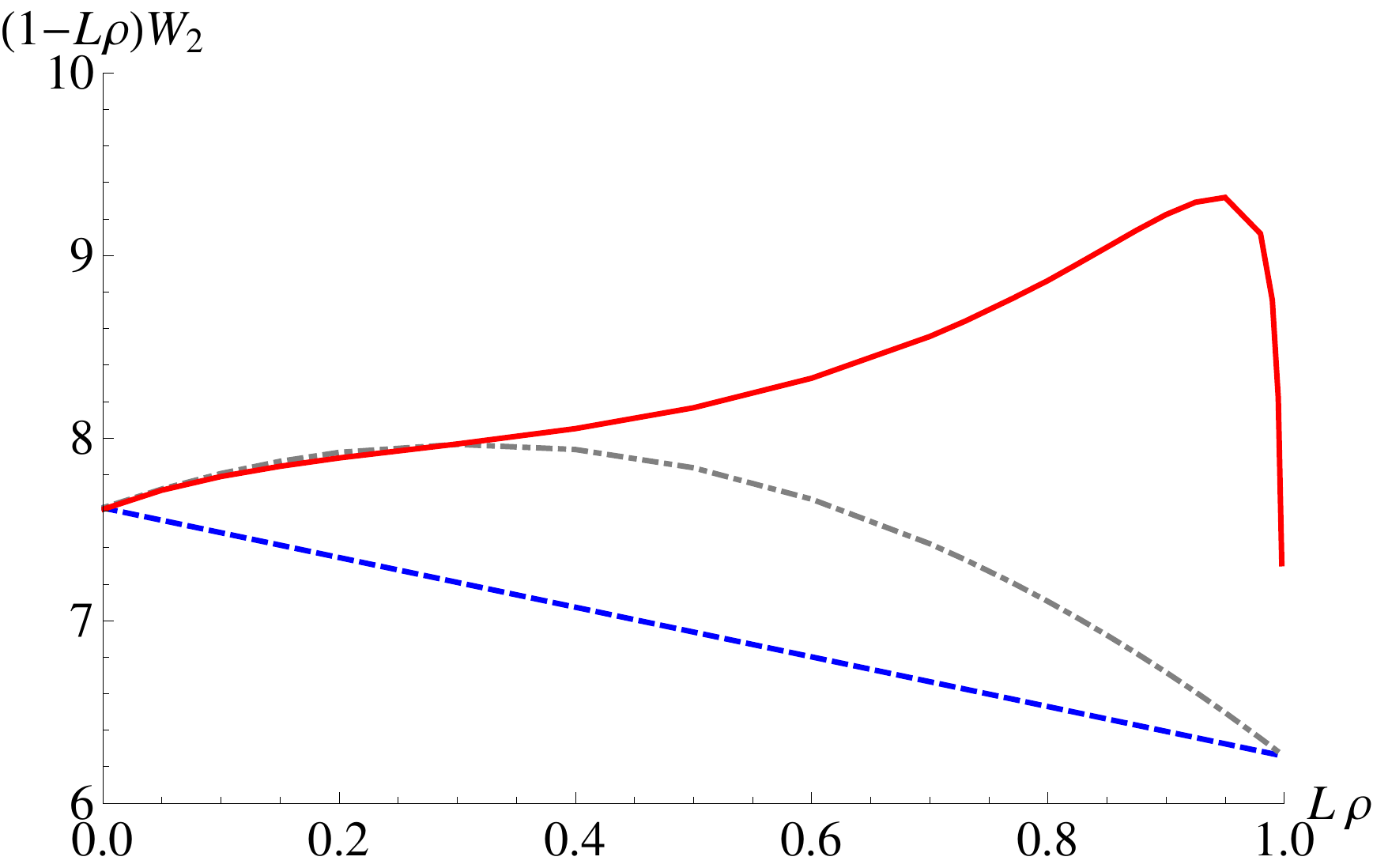}\\
(f)
\end{center}
}
\hfill
\mbox{}\\
\begin{center}
\includegraphics[width=0.3\linewidth]{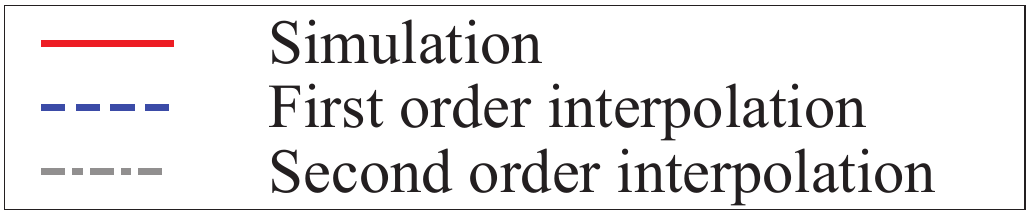}\\
\end{center}
\caption{Three plots of mean waiting times $W_i$ and scaled mean waiting times $(1-L\rho)W_i$, taken from Example 2 in Section \ref{numericalresults}. Figures (a) and (b) correspond to flow 8 of Intersection 2. Figures (c) and (d) correspond to flow 3 of Intersection 1. Figures (e) and (f) correspond to flow 2 of Intersection 3.}
\label{accuracypictures}
\end{figure}
We end this example with a final note on the accuracy of the approximations \eqref{ewapprox1} and \eqref{ewapprox2}. As stated before, the second-order interpolation \eqref{ewapprox2} generally gives better results, unless Criterion \eqref{firstordercriterion} is satisfied. In Figure \ref{accuracypictures}(a) the mean simulated delay for flow 8 of Intersection 2 is plotted, as well as the approximations \eqref{ewapprox1} and \eqref{ewapprox2}. We have chosen this particular example, because in Intersection 2 none of the flows satisfy Criterion \eqref{firstordercriterion}, and no group contains more than one flow with a very high flow ratio. In this situation there is no need to use the first-order interpolation, because it only gives worse results, as can be seen in Figure \ref{accuracypictures}(a). Figure \ref{accuracypictures}(b) shows the scaled mean delay $(1-L\rho)W_i$, also for the simulated delay and the two approximations. Clearly, the first-order interpolation should not be used here. For flow 3 in Intersection 1, Criterion \eqref{firstordercriterion} is satisfied. In Figure \ref{accuracypictures}(c) and \ref{accuracypictures}(d) we see that the non-linearity for small values of $L\rho$ results in an underestimation of the actual delay by the second-order interpolation. The first-order interpolation is more suitable here. Finally, Figures \ref{accuracypictures}(e) and \ref{accuracypictures}(f) show why it is sometimes impossible to find a simple polynomial that describes the behaviour of the scaled delay well. These figures, taken from Intersection 3, show that the HT limit is reached only extremely late $(L\rho > 0.99)$ for flow 2. In fact, for $L\rho=0.99$ there is still a gap between the simulated value and the HT limit. Due to the fact that the HT limit is being approached so slowly, the approximations are not very accurate in the range $0.8 < L\rho < 0.99$, with relative errors greater than 20\% for all flows.
Neither the first-order interpolation, nor the second-order interpolation is complex enough to describe the behaviour of $(1-L\rho)W_i$. These pictures also indicate that a fitting function more sophisticated that a first or second order polynomial is required if one wants to obtain a more accurate closed-form approximation.

\subsection*{Example 3: The impact of the stay-empty assumption}

Throughout the paper we have assumed that vehicles arriving during a green period while their flow is already empty, experience no delay at all because they do not have to accelerate and cross the intersection at normal speed. This assumption, which we refer to as the ``stay-empty assumption'', has been made because it makes the model more realistic than a standard queueing model with queues emptying and possibly refilling during the same green period. In this example we study the impact of this assumption by comparing delays found in the previous example to delays of vehicles in the same intersections, but assuming that queues do not stay empty. Before we carry out the numerical analysis, we discuss the model with refilling queues in more detail. The derivation of the LT limit \eqref{EWlt} uses the stay-empty assumption only in Equation \eqref{ewLT2}. The LT limit of the model with refilling queues is obtained by replacing \eqref{ewLT2} with $\E[W_{g,j}^{(G_{g,l})}] = \E[B_{g,j}] +\O(\rho)$ for $l\neq j$. This would, by the way, slightly simplify the LT limit \eqref{EWlt} and, hence, also simplify the second-order interpolation \eqref{ewapprox2}. Another interesting observation is that the distributions of the scaled delays under HT conditions do not change at all. Without providing a rigorous proof here, we argue that the HT fluid limits remain exactly the same if the stay-empty assumption is abandoned. Consequently, the stability condition $\sum_{g=1}^M\rho_{g,1}<1$ does not change, which can also be proven by making minor modifications to the proof in Appendix \ref{stabilityproof}. Note that the first-order interpolation for the stay-empty model is exactly the same as for the model with refilling queues.
This means that the stay-empty assumption, perhaps surprisingly, does not have much impact at all on the simulated and the approximated mean delays. We show one example of how close the mean delays for the two different models are. In Figure~\ref{stayemptyfig} we study the mean delays of vehicles in flow~3 of Intersection~1 again, just like in Figure~\ref{accuracypictures}(c).
In Figure~\ref{stayemptyfig} the mean delays are plotted against $L\rho$ for the model with and without the stay-empty assumption. As a reference, the approximated mean values are also shown in the same figure. Since we used a first-order interpolation for this example, the approximations for the model with and without the stay-empty assumption are the same. The relative difference between the simulated values for the two models is at most 7\%, for $L\rho=0.98$. Summarising, the approximation for the mean delay can be used for models with and without the stay-empty assumption, although it is recommended to adapt the LT limit slightly as stated before. However, one should keep in mind that scaled mean delay for the model with refilling queues converges to the HT limit slower than for the model with flows that stay empty.
\begin{figure}[h!]
\mbox{}\hfill
\includegraphics[width=0.5\textwidth]{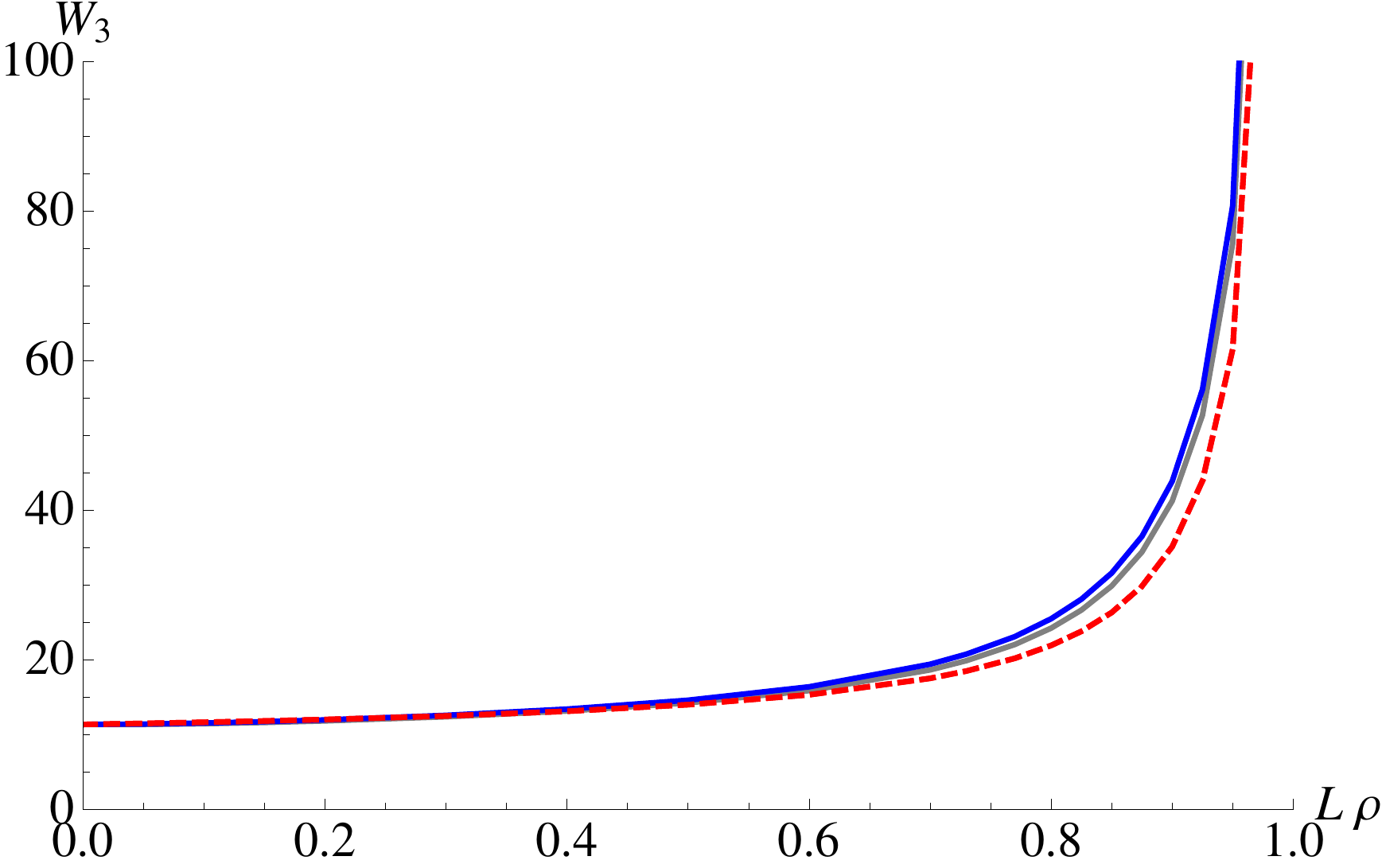}
\hfill
\includegraphics[width=0.4\textwidth]{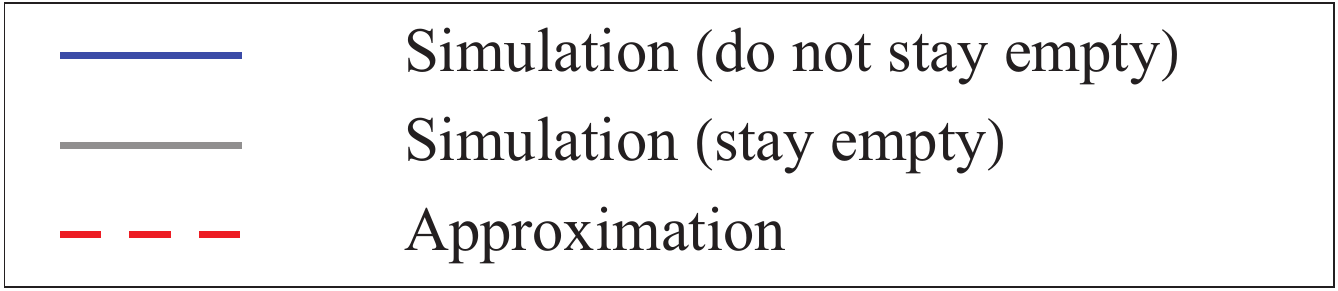}
\hfill\mbox{}
\caption{The mean delays for flow 3 of Intersection 1, for the model with and without the stay-empty assumption, compared to the approximated mean delay.}
\label{stayemptyfig}
\end{figure}

\section{Conclusions and topics for future research}\label{conclusionssection}

Under LT and HT conditions, we managed to accurately describe the behaviour of delay of vehicles approaching a traffic intersection with an exhaustive control policy. Based on these limiting situations, mean delay approximations have been established for any workload. These approximations are easy to implement, and have been tested on real-life situations. The performance of the approximations is good, except when the two greatest flow ratios within a group are very close. 

Several extensions of the model considered in the present paper are interesting to study.
\begin{itemize}
\item In order to build a better model for intersections that are part of an arterial system, it would be interesting to allow correlated batch arrivals: groups of vehicles arriving simultaneously. The studies by Levy and Sidi \cite{levysidi91} and Van der Mei \cite{vdmeipollingbatch2001} on polling systems with correlated batch arrivals might prove useful in this respect.
\item Furthermore, using results on polling systems with priorities (see, e.g., \cite{boonadanboxma2queues2010jdeds,boonpriorities}), one could also analyse \emph{conflicting} flows receiving a green light simultaneously. The conflicting flows should be placed in the same lane group and by assigning priority levels to each of the conflicting flows, the right of way rules can be implemented. Only minor adaptations to the LT and HT limits are required to introduce the priority levels.
\item It may also be possible to extend the \textit{distributional} approximations for polling systems developed by Dorsman et al. \cite{dorsman2010} to traffic intersections. The two required ingredients of such a distributional approximation are both derived in the present paper. That is, we need an HT approximation for the waiting-time distribution (as derived in Section \ref{HTsection}) in conjunction with a mean waiting-time approximation for general load (as obtained in Section \ref{interpolationssection}). Subsequently, the HT distributional approximation should be refined such that its mean coincides with the mean waiting-time approximation, while the resulting distributional approximation keeps the correct limiting behavior in HT after refinement. 
\item From a practical point of view, it would be desirable to have a model that allows flows to be part of multiple groups, instead of just one. However, this would complicate the analysis considerably at some points, because it is not straightforward anymore to determine which flows are dominant in each group.
\item Finally, possibly the greatest challenge, is the analysis of intersections with time-limited service. Although these are by far the most commonly used intersections in practice, hardly any detailed results are known to reliably estimate the expected delays. 
\end{itemize}

\section*{Acknowledgements}

The authors wish to thank Onno Boxma for valuable discussions and for useful comments on earlier drafts of the present paper.

\appendix
\section*{Appendix}
\section{Proof of stability}\label{stabilityproof}

In this appendix, we provide a rigorous proof of Theorem \ref{stabilitytheorem}. This theorem states that the stability condition of an intersection with exhaustive traffic control only depends on the flow ratios of the dominant flows in each group.

\subsection{Model}
For each flow $\{g,j\}$, $g =1,\dots,M$ and $j=1,\dots,N_g$, there is an i.i.d. sequence of interarrival times $\{A_{g,j,k}\}$ and an i.i.d.\ sequence of headways $\{ B_{g,j,k} \}$. All sequences are mutually independent. We assume that each sequence has a finite first moment and define the appropriate rates:
\begin{eqnarray*}
\lambda_{g,j} &=& 1/\mathbb{E}[A_{g,j,1}],\\
\mu_{g,j} &=& 1/\mathbb{E}[B_{g,j,1}].
\end{eqnarray*}
We assume that the interarrival time distribution is unbounded and spread-out, in other words $\P( A_{g,j,1} > T) >0$ for all $T>0$, and for some integer $\ell$,
\[
\mathbb{P}(A_{g,j,1}+\cdots + A_{g,j,\ell} \in \dd x) \geq q(x)\dd x
\]
where
\[
\int_0^\infty q(x)\dd x>0.
\]

The control policy is to allow flows $\{1,1\},\dots,\{1,N_1\}$ to operate until all of them are empty, then allow flows $\{2,1\},\dots,\{2,N_2\}$ to operate until they are all empty and so on, until flows $\{M,1\},\dots,\{M,N_M\}$ have operated. Then, this cycle is repeated. We assume that when switching for the $m$th time from group 1 to group 2, there is an all-red period of $R_{1,m}$. In general, for the $m$th switch from group $g$ to group $g+1$ (all group indices are modulo $M$), the all-red time is given by $R_{g,m}$. We assume that for $g=1,\dots,M$, $\{ R_{g,m} \}$ are i.i.d.\ sequences with $\mathbb{E}[R_{g,1}]<\infty$.

Let $E_{g,j}(t)$ be the number of vehicles arriving to flow $\{g,j\}$ in $(0,t]$. Let $S_{g,j}(t)$ be the \emph{potential} number of service completions by flow $\{g,j\}$ in $(0,t]$, i.e.\  $S_{g,j}(t)$ is the number of vehicles that would have departed from flow $\{g,j\}$ between times 0 and $t$ if there were no idling of server $\{g,j\}$. Let $A^\textit{res}_{g,j}(t)$ be the residual interarrival time at time $t$ for stream $\{g,j\}$ and let $B^\textit{res}_{g,j}(t)$ be the residual headway of the vehicle crossing at flow $\{g,j\}$. We assume that $A^\textit{res}_{g,j}(t)$ and $B^\textit{res}_{g,j}(t)$ are right continuous. Thus we have
\begin{eqnarray*}
E_{g,j}(t) & = & \max \{ n\geq 0:A^\textit{res}_{g,j}(0)+A_{g,j,1}+\cdots +A_{g,j,n-1} \leq t \}, \\
S_{g,j}(t) & = & \max \{ n\geq 0:B^\textit{res}_{g,j}(0)+B_{g,j,1}+\cdots +B_{g,j,n-1} \leq t \}.
\end{eqnarray*}
where the maximum of an empty set is defined to be zero. Let $T_{g,j}(t)$ be the cumulative busy time for server $\{g,j\}$ in $(0,t]$. Then the number at server $\{g,j\}$, $Q_{g,j}(t)$ at time $t$, is
\begin{equation} \label{q}
Q_{g,j}(t)=Q_{g,j}(0)+E_{g,j}(t)-S_{g,j}(T_{g,j}(t)),
\end{equation}
where $T_{g,j}(t)$ is determined by the control policy. Define
\[
X(t)=(Q_{g,j}(t),A_{g,j}(t),B_{g,j}(t),I(t):g=1,\dots,M,~j=1,\dots,N_g),
\]
where $I(t)$ is the group number $(1,\dots,M)$ that receives a green light at time $t$ (it can be set to an arbitrary value during the all-red times). Then it is not difficult to see that $X= \{ X(t)\}$ is a Markov process and has the strong Markov property. The assumption that the interarrival time distribution is unbounded and spread-out allows us to conclude that the states where $Q_{g,j}(t)=0$ are reachable and to show the existence of a continuous component for $X$, see Meyn and Down \cite{meydow94}.

\subsection{Stability of fluid models}
Let $q=\sum_{g=1}^M\sum_{j=1}^{N_g}Q_{g,j}(0)$. Suppose that the function $(\bar{Q}_{g,j}(\cdot),\bar{T}_{g,j}(\cdot):g=1,\dots,M,\ j=1,\dots,N_g)$ is a limit point of
\[
(q^{-1}Q_{g,j}(qt),q^{-1}T_{g,j}(qt):g=1,\dots,M,\ j=1,\dots,N_g),
\]
when $q\rightarrow \infty$. When it exists $(\bar{Q}_{g,j}(t),\bar{T}_{g,j}(t):g=1,\dots,M,\ j=1,\dots,N_g)$ is called a \emph{fluid limit} of the system. Since we have been able to describe the system dynamics in the form (\ref{q}), Dai \cite[Theorem~2.3.2]{dai99} yields that a fluid limit exists (it may not be unique). Furthermore, each of the fluid limits is a solution of the following set of conditions, known as the {\it fluid model}.
\begin{eqnarray*}
\bar{Q}_{g,j}(t) &=& \bar{Q}_{g,j}(0)+\lambda_{g,j}(t)-\mu_{g,j}\bar{T}_{g,j}(t)\\
\bar{Q}_{g,j}(t) &\geq& 0\\
\bar{T}_{g,j}(0) &=& 0 \ \mbox{and $\bar{T}_{g,j}(\cdot)$ is non-decreasing}
\end{eqnarray*}
plus additional conditions on $\bar{T}_{g,j}(\cdot)$ due to the control policy. This last condition is usually the most important, even though it is vague at this point.

The fluid model is said to be {\it stable} if there exists a fixed time $t_0$ such that $\sum_{g=1}^M\sum_{j=1}^{N_g}\bar{Q}_{g,j}(t)=0$, for all $t\geq t_0$, for any solution of the fluid model. The fluid model is said to be {\it unstable} if for every solution of the fluid model with $\sum_{g=1}^M\sum_{j=1}^{N_g}\bar{Q}_{g,j}(0)=0$, there exists a $\delta>0$ such that $\sum_{g=1}^M\sum_{j=1}^{N_g}\bar{Q}_{g,j}(\delta) \neq 0$. Thus, stability of the fluid model states that eventually all flows will drain, and once drained, they remain empty.

The connections between the Markov process and the fluid model are as follows: if the fluid model is stable, there exists a unique invariant probability for $X$ (Theorem~4.2 of Dai \cite{dai95}). If the fluid model is unstable, then $X$ is transient (Theorem~2.5.1 of Dai \cite{dai99}). If one is interested in finiteness of moments, then under corresponding assumptions on the underlying random variables, stability of the fluid model yields finiteness of moments, see Dai and Meyn \cite{daimey95}. For example, assuming finite second moments on the underlying random variables, stability of the fluid model implies the existence of an invariant probability and finite mean queue lengths.

For the model under consideration, let $\rho_{g,j}=\lambda_{g,j}/\mu_{g,j}$ and assume without loss of generality that $\rho_{g,1} > \rho_{g,2} > \dots > \rho_{g,N_g}$ for all $g=1,\dots,M$. Then we have the following result.

\begin{proposition}
(i) If $\sum_{g=1}^M\rho_{g,1}<1$, the fluid model is stable and, thus, an invariant probability $\varphi$ exists for $X$.\\
(ii) If $\sum_{g=1}^M\rho_{g,1}>1$, then the fluid model is unstable and $X$ is transient.
\end{proposition}

\begin{proof}
 (i) First, we show that the all-red periods can be ignored in the stability analysis. Suppose that at time 0, group $g$ has just completed and group $g'=(g+1)\mod M + 1$ begins service, where $g'$ is such that $\bar{Q}_{g',j}(0)>0$ for some $j \in \{ 1,\ldots, N_{g'} \}$. We then have that there exists $\delta>0$ such that  $\bar{Q}_{g',j}(s)>0$ for $s\in [0,\delta]$. In this case, we have
\begin{eqnarray*}
\lefteqn{\lim_{q\rightarrow \infty} \frac{T_{g,j}(\delta q)}{q}}\\
& = & \lim_{q\rightarrow \infty} \frac{\max (\delta q-R_g,0)}{q}\\
& = & \lim_{q\rightarrow \infty} \left( \max \left( \delta-\frac{R_g}{q},0 \right) \right) \\
& = & \delta
\end{eqnarray*}
from which we have $\frac{d}{dt} \bar{T}_{g,j}(t) |_{t=0} = 1$ and thus without loss of generality we can assume that the all-red times are zero.

Next, we show that if we are serving group $g$ at time $t$ and $\bar{Q}_{g,j}(t)=0$ for all $j\in \{ 1,\ldots,N_g \}$, then for any fluid limit, we immediately switch to a new group (unless for all $g,j$, $\bar{Q}_{g,j}(t)=0$). Note that for group $g$, queues $1,\ldots,N_g$ operate as $N_g$ stable queues in parallel (until all queues are simultaneously empty), which is a special case of a stable generalized Jackson network (as $\rho_{g,j} < 1$ for all $j$). Denote by $T_0^{(n_1,\ldots,n_{N_g})}$ the time to reach the state when all queues are empty, starting from the initial condition is that there are $n_i$ jobs at queue $i$. Theorem~3.8 of \cite{meydow94} implies that all queues being empty is reachable and as
\[
\frac{Q_{g,j}(qt)}{q} \rightarrow 0,
\]
then
\[
\frac{T_0^{(Q_{g,1}(qt),Q_{g,2}(qt),\ldots,Q_{g,N_g}(qt))}}{q} \rightarrow 0
\]
and thus on the fluid scale, switching away from group $g$ is immediate when all of the queues are empty (on the fluid scale).

Next, we show that from an arbitrary initial condition for the fluid model (i.e.\ $\bar{Q}_{g,j}(0) = x_{g,j} \geq 0$) and with the server initially at group 1, we have that group $i$ is switched away from for the first time at time $t_i$, where $t_i$ is given by:
\begin{eqnarray*}
t_1 & = & \max_{1\leq j \leq N_1} \left\{ \frac{x_{1,j}}{\mu_{1,j}-\lambda_{1,j}} \right\} \\
t_2 & = & t_1 + \max_{1\leq j \leq N_2} \left\{ \frac{x_{2,j}+\lambda_{2,j}t_1}{\mu_{2,j}-\lambda_{2,j}} \right\} \\
& \vdots \\
t_M & = & t_{M-1} + \max_{1\leq j \leq N_M} \left\{ \frac{x_{M,j}+\lambda_{M,j}t_{M-1}}{\mu_{M,j}-\lambda_{M,j}} \right\}.
\end{eqnarray*}
Clearly, if $\lambda_{g,j} / \mu_{g,j} < 1$ for all $g,j$, then $t_M < \infty$. However, at this point, by definition, it is clear that for $t\geq t_M$,
\begin{equation} \label{e:dominant}
\frac{\bar{Q}_{g,1}(t)}{\mu_{g,1}} \geq \frac{\bar{Q}_{g,j} (t)}{\mu_{g,j}}
\end{equation}
for all $g$ and $j \in \{ 2,\ldots,N_g \}$.

Now, consider
\[
V(t)=\sum_{g=1}^M \frac{\bar{Q}_{g,1}(t)}{\mu_{g,1}}.
\]
We trivially see that $\sum_{g=1}^M \frac{d}{dt} \bar{T}_{g,1}(t) = 1$ if $V(t)>0$, and thus for $t \geq t_M$ and if $V(t)>0$, using (\ref{e:dominant}),  $\frac{d}{dt} V(t)=\sum_{g=1}^M \rho_{g,1} - 1 < 0$. Therefore, there exists a $T<\infty$ such that $V(t)=0$ for all $t \geq T$ and again using (\ref{e:dominant}), we conclude that for all $g,j$ $\bar{Q}_{g,j}(t) = 0$ for all $t\geq T$.

This completes the proof of (i).

To show (ii), as $\sum_{g=1}^M \frac{d}{dt} \bar{T}_{g,1}(t) \leq 1$, we see that $\frac{d}{dt} V(t) \geq \sum_{g=1}^M \rho_{g,1} -1$, which is strictly positive and thus the fluid model is unstable.
\end{proof}

\section{Input settings for Example 2}\label{appendixintersections}
In this appendix we give the input settings for the intersections discussed in Section \ref{numericalresults}, Example 2.
\begin{center}
\begin{tabular}{|l|l|}
\multicolumn{2}{c}{\rule{0.9\textwidth}{0mm}} \\
\hline
\multicolumn{2}{|c|}{Intersection 1} \\
\hline
$N$ & 9 \\
\hline
Vehicle types & 5 flows for cars $(1,\dots,5)$, 4 flows for bicycles $(6,\dots,9)$\\
\hline
Arrival processes & Poisson \\
\hline
Arrival intensities & 280, 930, 700, 120, 240, 60, 60, 60, 60 \\
(cars/bikes per hour) & \\
\hline
Saturation flow rates  & 1800, 1900, 1900, 1700, 1700,  \\
(cars/bikes per hour) & 10000, 10000, 10000, 10000\\
\hline
SCVs of the headways & \rule[-1.5ex]{0cm}{4ex}$\C^2_{B_i}=1$ for cars $(i \leq 5)$, and $\C^2_{B_i}=0$ for bikes $(i\geq6)$\\
\hline
Groups & $\{2,3,8,9\}, \{4\}, \{6,7\}, \{1,5\}$ \\
\hline
All-red times & $2,8,4,5$\\
\hline
\end{tabular}

\noindent\begin{tabular}{|l|l|}
\multicolumn{2}{c}{\rule{0.9\textwidth}{0mm}} \\
\hline
\multicolumn{2}{|c|}{Intersection 2} \\
\hline
$N$ & 11 \\
\hline
Vehicle types & 7 flows for cars $(1,\dots,7)$, 4 flows for bicycles $(8,\dots,11)$\\
\hline
Arrival processes & Poisson \\
\hline
Arrival intensities & 263, 344, 332, 381, 148, 442, 258, 60, 60, 60, 60 \\
(cars/bikes per hour) & \\
\hline
Saturation flow rates & 1950, 1950, 1950, 1800, 1700, 1950, 1700,  \\
(cars/bikes per hour) & 10000, 10000, 10000, 10000\\
\hline
SCVs of the headways & \rule[-1.5ex]{0cm}{4ex}$\C^2_{B_i}=1$ for cars $(i \leq 7)$, and $\C^2_{B_i}=0$ for bikes $(i\geq8)$\\
\hline
Groups & $\{1,3,9,11\}, \{2,5\}, \{4,8\}, \{6,7,10\}$ \\
\hline
All-red times & $8,1,4,6$\\
\hline
\end{tabular}
\begin{tabular}{|l|l|}
\multicolumn{2}{c}{\rule{0.9\textwidth}{0mm}} \\
\hline
\multicolumn{2}{|c|}{Intersection 3} \\
\hline
$N$ & 10 \\
\hline
Vehicle types & 6 flows for cars $(1,\dots,6)$, 4 flows for bicycles $(7,\dots,10)$\\
\hline
Arrival processes & Poisson \\
\hline
Arrival intensities  & 680, 150, 390, 860, 280, 430, 100, 100, 100, 100 \\
(cars/bikes per hour) & \\
\hline
Saturation flow rates  & 1950, 1700, 1850, 1950, 1700, 1850,  \\
(cars/bikes per hour) & 10000,10000,10000, 10000\\
\hline
SCVs of the headways & \rule[-1.5ex]{0cm}{4ex}$\C^2_{B_i}=1$ for cars $(i \leq 6)$, and $\C^2_{B_i}=0$ for bikes $(i\geq7)$\\
\hline
Groups & $\{1,4,8,10\}, \{2,5\}, \{3,6,7,9\}$ \\
\hline
All-red times & $4,2,5$\\
\hline
\end{tabular}
\end{center}

\expandafter\ifx\csname urlstyle\endcsname\relax
  \providecommand{\doi}[1]{DOI: #1}\else
  \providecommand{\doi}{DOI: \begingroup \urlstyle{rm}\Url}\fi

\bibliographystyle{abbrvnat}

\end{document}